\documentclass[12pt]{amsart}
\usepackage[margin=1in]{geometry}
\usepackage{amsmath, amssymb, amsthm,hyperref,tikz-cd}
\usepackage{mathtools,xcolor}
\usepackage{enumitem}
\usepackage{multicol}

\usepackage{soul}

\newtheorem{theorem}{Theorem}[section]
\newtheorem{lemma}[theorem]{Lemma}
\newtheorem{proposition}[theorem]{Proposition}
\newtheorem{corollary}[theorem]{Corollary}
\theoremstyle{definition}
\newtheorem{definition}[theorem]{Definition}
\newtheorem{example}[theorem]{Example}

\newtheorem{conjecture}[theorem]{Conjecture}
\newtheorem{remark}[theorem]{Remark}

\newcommand{\kk}{\bold k}

\newcommand{\g}{\mathfrak{g}}
\newcommand{\n}{\mathfrak{n}}

\newcommand{\norm}[1]{\left\lVert#1\right\rVert}

\numberwithin{equation}{section} 
\usepackage{color}

\begin{document}

\title[Periodic and quasi-motivic pencils of flat connections]{Periodic and quasi-motivic pencils of flat connections}

\author{Pavel Etingof}

\address{Department of Mathematics, MIT, Cambridge, MA 02139, USA}

\author{Alexander Varchenko}

\address{Department of Mathematics, University of North Carolina at Chapel Hill, 
CB\# 3250 Phillips Hall
Chapel Hill, N.C. 27599, USA} 

\maketitle

\centerline{\bf To the memory of Igor Krichever} 

\begin{abstract} We introduce a new notion of a {\bf periodic pencil of flat connections} on a smooth algebraic variety $X$. This is a family $\nabla(s_1,...,s_n)$ of flat connections on a trivial vector bundle on $X$ depending linearly on parameters $s_1,...,s_n$ and generically invariant, up to isomorphism, under the shifts $s_i\mapsto s_i+1$ for all $i$. If  
in addition $\nabla$ has regular singularities, we call it a {\bf quasi-motivic pencil}. We 
use tools from complex analysis to establish various remarkable properties of such pencils over $\Bbb C$.  For example, we show that the monodromy of a quasi-motivic pencil is defined over the field of algebraic functions in $e^{2\pi is_j}$, and that its singularities are 
 constrained to an arrangement of hyperplanes with integer normal vectors. Then we show that many important examples of families of flat connections, such as Knizhnik-Zamolodchikov, Dunkl, and Casimir connections, are quasi-motivic and thus periodic pencils. 
 
Besides being interesting in its own right, the periodic property of a pencil of flat connections turns out to be very useful in computing the eigenvalues of the $p$-curvature of its reduction to positive characteristic. This is done in the follow-up paper \cite{EV1}. 
\end{abstract} 

\tableofcontents

\section{Introduction} 

In this paper we introduce a new notion of a {\bf periodic pencil of flat connections}. 
Namely, a {\bf pencil  of flat connections} is a family 
$$
\nabla(\bold s)=d-s_1B_1-...-s_nB_n
$$
of flat connections on the trivial rank $N$ vector bundle 
over a smooth irreducible variety $X$ over a field $\kk$, where $B_j\in \Omega^1(X)\otimes {\rm Mat}_N(\kk)$ for $1\le j\le n$. 
Such a pencil is said to be {\bf periodic}\footnote{In fact, we introduce this definition in a larger generality of {\bf families of flat connections}, where the dependence on $s_j$ is allowed to be polynomial, but our main results only apply to pencils, where this dependence is linear.} if it admits {\bf shift operators} 
$$
A_j\in GL_N(\kk(\bold s)[X])
$$ 
such that 
$$
\nabla(\bold s+\bold e_j)\circ A_j(\bold s)=A_j(\bold s)\circ \nabla(\bold s),\ 1\le j\le n.
$$
There are many interesting examples of periodic pencils, falling into three large classes: Knizhnik-Zamolodchikov (KZ) connections, Dunkl connections, and Casimir connections. Another important source of examples (often belonging to one of these classes) is equivariant quantum connections of conical symplectic resolutions of singularities with finitely many torus fixed points.

At the same time, periodic pencils enjoy many remarkable properties, especially when they have regular singularities. 
We call periodic pencils with regular singularities  {\bf quasi-motivic}. For example, we show that quasi-motivic pencils give rise to local systems defined over $\overline{\Bbb Q}$ in both de Rham and Betti realizations, and their singularities can only occur on an arrangement of hyperplanes. 
   
  In a follow-up paper \cite{EV1}, we describe the spectrum of the $p$-curvature of a periodic pencil in characteristic $p$. When applied to equivariant quantum connections for conical resolutions with finitely many torus fixed points, this has applications to symplectic geometry, since it has been shown in \cite{Lee} that the $p$-curvature of such a connection coincides with the equivariant version of Fukaya's quantum Steenrod operation (at least up to a nilpotent correction that does not affect the spectrum).

The organization of the paper is as follows. In Section 2 we discuss preliminaries and auxiliary results. In Section 3 we discuss properties of periodic pencils of flat connections 
and of {\bf pencils with periodic monodromy}, which is a weaker but more readily verifiable condition that makes sense over $\Bbb C$. This condition requires that the monodromy 
representation of $\nabla(\bold s)$ up to isomorphism generically depend only on the exponentials $q_j:=e^{2\pi is_j}$. The main result in this section is Theorem \ref{permoncrit}, which states that the condition of periodic monodromy is equivalent to the 
monodromy of $\nabla(\bold s)$ being defined over a finite Galois extension of the field 
$\Bbb C(q_1,...,q_n)$ and stable under the Galois group of this extension. 
We also explain that equivariant quantum connections of conical symplectic resolutions 
with finitely many torus fixed points give rise to periodic pencils. 

In Section 4 we introduce the notion of a {\bf quasi-motivic pencil}, which 
is a pencil with periodic monodromy and regular singularities. We show in Theorem \ref{qmper} that such a pencil is necessarily periodic. Then we proceed to study 
properties of quasi-motivic pencils, showing that their singularities in $\bold s$ can only 
occur on translates of hyperplanes in $\Bbb C^n$ defined over $\Bbb Q$. 
We also discuss {\bf motivic} pencils, which are a large special class 
of quasi-motivic pencils arising as Gauss-Manin connections.
Finally, we discuss the notion of a {\bf quasi-geometric} local system, which is a local system defined over $\overline{\Bbb Q}$ both in the de Rham and the Betti realization. It is a generalization of the notion of a geometric local system due to Deligne (\cite{De}). 
We explain how quasi-motivic pencils give rise to quasi-geometric local systems, 
and then make a conjecture that local systems arising from braided fusion categories 
are quasi-geometric, discussing some evidence in support of this conjecture. 

Finally, in Section 5 we review numerous examples of quasi-motivic (hence periodic) pencils, and discuss their irreducibility, unitarity and generic semisimplicity.  
 
{\bf Acknowledgements.} This paper belongs to the area of quantum integrable systems which took shape in the work 
of our friend and colleague Igor Krichever (1950-2022). 
Igor influenced us in many ways throughout our lives, 
and we dedicate this paper to his memory.
 
We thank Jae Hee Lee and Vadim Vologodsky for useful discussions. P. E.'s work was partially supported by the NSF grant DMS-2001318
and  A. V.'s work was partially supported by the NSF grant DMS-1954266. 

\section{Preliminaries} 

\subsection{Galois-stability} \label{galstab} 
Let $K$ be a field of characteristic $0$. Let $G$ be an affine algebraic group and $Y$ an affine $G$-variety, both defined over $K$. 

\begin{definition} 
A point $y\in Y(\overline{K})$ is said to be {\bf Galois-stable} up to $G$-action if for any $\gamma\in {\rm Gal}(\overline K/K)$ there exists $g\in G(\overline K)$ such that $\gamma(y)=gy$. 
\end{definition}

It is clear that whether or not a point $y$ is Galois-stable up to $G$-action is determined by its $G(\overline K)$-orbit, and that 
if $gy\in Y(K)$ for some $g\in G(\overline K)$ then $y$ is  Galois-stable. The converse is false, however: 
e.g., take $K=\Bbb R$, $Y\subset \Bbb A^1$ 
defined by the equation $x^2=-1$ (so $Y(K)=\emptyset$), $G=\lbrace1,-1\rbrace$ acting on $x$ by multiplication, and $y=i$. So in general being Galois-stable up to $G$-action is a strictly weaker condition than being conjugate to a point defined over $K$. 

\begin{example}\label{holfam1} Let $A$ be a finitely generated 
$K$-algebra, $N$ a positive integer, $Y$ the variety of $N$-dimensional matrix representations of $A$, and $G=PGL_N$ acting on $Y$ by conjugation. Then a representation $\eta: A\to {\rm Mat}_N(\overline K)$ is Galois-stable up to $G$-action if and only if for any $\gamma\in {\rm Gal}(\overline K/K)$, the representations $\eta$ and $\gamma(\eta)$ are isomorphic. 
In particular, this notion applies to representations of finitely generated groups, by taking $A$ to be a group algebra. 
For brevity, from now on we will speak about Galois-stable representations, dropping the words ``up to $G$-action".
\end{example} 

Now let $G,Y$ be defined over $\Bbb C$ and let $K:=\Bbb C(\bold q)$, where $\bold q:=(q_1,...,q_n)$ is a collection of variables. Then for every $y\in Y(\overline{K})$ 
there exists a nonempty affine open subset $U\subset \Bbb C^n$ 
and a finite unramified Galois cover $p: \widetilde U\to U$ 
such that $y\in Y(\Bbb C[\widetilde U])$. 
Thus, given $\bold a:=(a_1,...,a_n)\in U$ and $\widetilde{\bold a}\in p^{-1}(\bold a)$,
we can define the specialization $y|_{\widetilde{\bold a}}\in 
Y(\Bbb C)$ of $y$ at $\widetilde{\bold a}$. However, if 
$\eta$ is Galois-stable up to $G$-action then the $G(\Bbb C)$-orbit of $y|_{\widetilde{\bold a}}$ depends only on $\bold a$. So when we only care about the $G(\Bbb C)$-orbit of this specialization, we will denote it by $y|_{\bold a}$. 

\subsection{Exponentially bounded holomorphic functions} 

Let $V$ be a finite dimensional complex vector space, 
and $f: \Bbb C^n\to V$ a holomorphic function. 
We say that $f$ is {\bf exponentially bounded} if there exists $C>0$ such that 
\begin{equation}\label{normcon} 
\norm{f(\bold s)}=O(e^{C\norm{s}}),\ \bold s\to \infty. 
\end{equation} 
For $\bold b=(b_1,...,b_n)\in \Bbb C^n$ let 
$e^{\bold b}:=(e^{b_1},...,e^{b_n})$.

\begin{lemma}\label{complan} Let $g,h: \Bbb C^n\to \Bbb C$ be exponentially bounded holomorphic functions 
such that $h\ne 0$. Suppose that the meromorphic function $g/h$ is periodic under the lattice $\Bbb Z^n$. Then $g/h$ is a rational function of $e^{2\pi i\bold s}$, $\bold s\in \Bbb C^n$. 
\end{lemma} 

\begin{proof} We first prove the statement for $n=1$. 
In this case by a well known theorem in complex analysis (\cite{SS}, Chapter 5, Theorem 2.1), the number of zeros of $g$ and $h$ in the disk of radius $R$ is $O(R)$ as $R\to \infty$. 
Thus so is the number of zeros and poles of $g/h$. But $g(s)/h(s)=f(e^{2\pi is})$
where $f$ is a meromorphic function on $\Bbb C^\times$. It follows that 
$f$ has finitely many zeros and poles. Also, since $g$ and $h$ are exponentially bounded, $f$ has at most power growth at zero and infinity. Thus $f$ is rational, as claimed. 

Now we prove the result in general. By the one-dimensional case, 
$f$ is rational in each variable when the other variables are fixed. This implies that $f$ is rational by the Hurwitz-Weierstrass theorem, \cite{BM}, 9.5, Theorem 5 (p.201). 
\end{proof} 

\subsection{Generically periodic holomorphic maps up to $G$-action} 

Let $\bold e_1,...,\bold e_n$ be the standard basis 
of the standard $n$-dimensional space. Let $G$ be an complex affine algebraic group and $V$ a finite dimensional algebraic representation of $G$. 
Let $\rho: \Bbb C^n\to V$ be a holomorphic map.

\begin{definition} We say that $\rho$ is {\bf generically periodic} up to $G$-action if for each $j$ the vectors $\rho(\bold s+\bold e_j)$ and $\rho(\bold s)$ are $G$-conjugate for an analytically Zariski dense set of $\bold s\in \Bbb C^n$. 
\end{definition} 

\begin{theorem}\label{holfam} Let $\rho: \Bbb C^n\to V$ be an exponentially bounded holomorphic map which is generically periodic up to $G$-action. Then there exists 
a nonempty affine open subset $U$ of $(\Bbb C^\times)^n$, a finite unramified Galois cover $p: \widetilde U\to U$, and 
a Galois-stable up to $G$-action regular function $\eta: \widetilde U\to V$ such that $\rho(\bold s)$ is $G(\Bbb C)$-conjugate to $\eta|_{e^{2\pi i\bold s}}$ for $e^{2\pi i\bold s}\in U$. Moreover, $\eta$ is unique up to action of $G(\overline{\Bbb C(\bold q)})$. 
\end{theorem} 

\begin{proof}  Consider the closed subvarieties $Y_i$ of $V$ of vectors $y\in V$ with stabilizer $G_y\subset G$ of codimension $\le i$. Thus $Y_0\subset Y_1\subset...\subset Y_{\dim G}=V$. Let $d$ be the smallest integer such that $\rho: \Bbb C^n\to Y_d$. 

Let $E\subset V$ be a generic affine subspace of codimension $d$ (a {\it gauge fixing subspace}). Then for sufficiently generic $\bold s$ (namely, outside of a proper analytic subset $Z\subset \Bbb C^n$) the set $\Xi_{\bold s}$ of $v\in E$ which are $G$-conjugate to $\rho(\bold s)$ has a finite cardinality $m$. 

So we obtain a meromorphic map $R: \Bbb C^n\to \Sigma^mV$
($m$-th symmetric power of $V$ as a variety) with poles at $Z$ given by $R(\bold s)=\Xi_{\bold s}$; i.e. for each regular function $\Psi$ on $\Sigma^mV$ the function $f_\Psi:=\Psi\circ R: \Bbb C^n\to \Bbb C$ is meromorphic with poles at $Z$. 

Moreover, since $\rho$ is generically periodic up to $G$-action, $R(\bold s+\bold e_j)=R(\bold s)$ and therefore $f_\Psi(\bold s+\bold e_j)=f_\Psi(\bold s)$ for an analytically Zariski dense set of $\bold s$. Thus $f_\Psi(\bold s+\bold e_j)=f_\Psi(\bold s)$ as meromorphic functions on $\Bbb C^n$. So $f_\psi(\bold s)=F_\Psi(e^{2\pi i\bold s})$, where $F_\Psi: (\Bbb C^\times)^n\to \Bbb C$ is a meromorphic function. 

Finally, by construction, $f_\Psi$ can be written as $g/h$, where $g,h: \Bbb C^n\to \Bbb C$ are polynomials of $\rho(\bold s)$, hence they are exponentially bounded. So it follows from Lemma \ref{complan} that $F_\Psi$ is rational. 
Thus $R$ descends to a rational map $\overline R: (\Bbb C^\times)^n\to \Sigma^mV$, which defines a point of 
$Q\in \Sigma^mV(\overline{\Bbb C(\bold q)})$.  
Now any $\eta\in Q$ is a Galois-stable element
of $\overline{\Bbb C(\bold q)}\otimes V$, and 
in fact belongs to $\Bbb C[\widetilde U]\otimes V$ 
for some $U\subset (\Bbb C^\times)^n$ and $p: \widetilde U\to U$, such that  
$\eta$ has the required property for $e^{2\pi i\bold s}\in U$.

It remains to show that $\eta$ is unique up to $G$-action. 
Let $\eta_1,\eta_2$ be two such elements. We may assume that both of them are defined over $\Bbb C[\widetilde U]$ where $U\subset (\Bbb C^\times)^n$ is a nonempty affine open subset and $p: \widetilde U\to U$ is a finite unramified Galois cover. 
Let $\bold a\in U$, $U_{\bold a}$ be a small ball around $\bold a$, and $U_1$ be a connected component of $p^{-1}(U_{\bold a})$. There exists a connected component 
$U_2$ of $p^{-1}(U_{\bold a})$ such that for an analytically dense set of $\bold b \in U_{\bold a}$ , $\eta_1|_{\bold b_1}=\eta_2|_{\bold b_2}$, where $\bold b_i$ are the preimages of $\bold b$ in $U_i$, $i=1,2$. So by Chevalley's elimination of quantifiers, 
there exists $g\in G(\overline{\Bbb C(\bold q)})$ such that 
$g\eta_1=\eta_2$.  
\end{proof} 

\subsection{Algebraic independence of exponentials} 
\begin{theorem}\label{Zil} 
Let 
$E: \Bbb C^n\to (\Bbb C^\times)^n$ be the exponential map 
given by 
$$
E(s_1,...,s_n):=(e^{2\pi is_1},...,e^{2\pi is_n}).
$$ 
Suppose that $Y\subset \Bbb C^n$ is an irreducible algebraic hypersurface 
such that $E(Y)$ is also contained in an algebraic hypersurface. Then $Y$ 
is a translate of a hyperplane defined over $\Bbb Q$. In other words, $E(Y)$ is an affine hypertorus 
in $(\Bbb C^\times)^n$. 
\end{theorem} 

\begin{proof} We may assume without loss of generality that $Y$ contains $0$, is smooth there, and is locally near $0$ parametrized by the coordinates $s_1,...,s_{n-1}$. 
Then $Y$ is defined near $0$ by the equation 
$$
s_n=F(s_1,...,s_{n-1}),
$$ 
where $F$ is an algebraic function. Assume that $Y$ is not a translate of a hyperplane defined over $\Bbb Q$. 
Then $s_1,...,s_n$ are analytic functions of $s_1,...,s_{n-1}$ 
linearly independent over $\Bbb Q$, and the Jacobi matrix $\frac{\partial (s_1,...,s_n)}{\partial (s_1,...,s_{n-1})}$ has rank $n-1$. Thus \cite{Ax}, Corollary 2, p.253 (for the complex field) implies that the transcendence degree of the field $\Bbb C(s_1,...,s_{n},e^{2\pi is_1},...,e^{2\pi is_{n}})$, is at least $2n-1$. So $e^{2\pi is_1},...,e^{2\pi is_{n}}$ must be algebraically independent over $\Bbb C(s_1,...,s_n)$ (as it has transcendence degree $n-1$), hence they are algebraically independent over $\Bbb C$. 
\end{proof} 

\subsection{Geometric local systems} \label{geoloc} In this subsection we briefly review 
Deligne's theory of geometric local systems (which we do not use here, and recall solely for motivational purposes). Let $X$ be a smooth irreducible complex variety. 
The following definition is due to Deligne (see \cite{De}). 

\begin{definition} A Betti $\overline{\Bbb Q}$-local system $\rho$ on $X$ is said to be {\bf geometric} if there exists a non-empty Zariski open subset $X^\circ \subset X$ and 
a smooth proper surjective morphism $\pi: Y\to X^\circ$
such that $\rho|_{X^\circ}$ is a direct summand 
in the monodromy representation of the Gauss-Manin connection on $H^i(\pi^{-1}(\bold x),\overline{\Bbb Q})$, 
$\bold x\in X^\circ$, for some $i$.
\end{definition} 

\begin{remark} 1. If $\nabla$ has finite monodromy group $G$ then it is geometric. 
Indeed, $\nabla$ is a direct summand in a multiple of $\pi_*\nabla_{\rm triv}$, where $\nabla_{\rm triv}$ is the trivial connection on the trivial line bundle on a Galois cover $\pi: \widetilde X\to X$ with Galois group $G$, which implies the statement. 

2. By the Decomposition Theorem, any geometric system $\rho$ is semisimple.

3. The category of geometric systems is invariant under pullbacks, external products (hence tensor products), and duals. Thus for every $X$, we have the semisimple $\overline{\Bbb Q}$-linear Tannakian category $\mathcal G_{\rm B}(X)$ of geometric Betti local systems on $X$, which is a tensor subcategory of the Tannakian category $\mathcal L_{\rm B}(X,\overline{\Bbb Q})$ of all 
Betti $\overline{\Bbb Q}$-local systems on $X$. Similarly, 
the extension of scalars $\mathcal G_{\rm B}(X)_{\Bbb C}$
is a tensor subcategory of the Tannakian category $\mathcal L_{\rm B}(X,\Bbb C)$ of all 
Betti $\Bbb C$-local systems on $X$.

4. The same definition can be used in the de Rham realization (for $\mathcal O$-coherent $D$-modules with regular singularities). Thus for every $X$, we also have the semisimple $\Bbb C$-linear Tannakian category $\mathcal G_{\rm dR}(X)$ of geometric de Rham local systems on $X$, which is a tensor subcategory of the Tannakian category $\mathcal L_{\rm dR}(X)$ of all de Rham local systems on $X$ with regular singularities. Moreover, if $X$ is defined over $\overline{\Bbb Q}$ then we have the full category $\mathcal L_{\rm dR}(X,\overline{\Bbb Q})\subset \mathcal L_{\rm dR}(X)$ of all de Rham local systems on $X$ with regular singularities defined over $\overline{\Bbb Q}$. 

5. The Riemann-Hilbert correspondence (taking monodromy of a regular flat connection)
 sets up a canonical tensor equivalence $\mathcal L_{\rm dR}(X)\cong \mathcal L_{\rm B}(X,{\Bbb C})$ (\cite{De1}), which in general (for $X$ defined over $\overline{\Bbb Q}$) {\bf does not} map the full subcategory $\mathcal L_{\rm dR}(X,\overline{\Bbb Q})$ into $\mathcal L_{\rm B}(X,\overline{\Bbb Q})_{\Bbb C}$. In other words, the monodromy of a holonomic system of linear  differential equations defined over $\overline{\Bbb Q}$ need not be defined over $\overline{\Bbb Q}$. However, the Riemann-Hilbert correspondence {\bf does} restrict to an equivalence between the subcategories $\mathcal G_{\rm dR}(X)$ and $\mathcal G_{\rm B}(X)_{\Bbb C}$, i.e., the monodromy of a direct summand of a Gauss-Manin connection is always defined over $\overline{\Bbb Q}$. In particular, 
every object of $\mathcal G_{\rm dR}(X)$ can be found in the de Rham cohomology of 
$\pi^{-1}(\bold x)$ where $\pi: Y\to X^\circ$ is defined over $\overline{\Bbb Q}$, hence 
is itself defined over $\overline{\Bbb Q}$. This striking property follows from the fact that monodromy of a geometric de Rham local system can be computed in the Betti realization by following the deformation of twisted cycles in $Y_{\bold x}$ as $\bold x$ varies along $X^\circ$. 
 \end{remark} 

\section{Periodic families and pencils of flat connections}

\subsection{Polynomial families and pencils of flat connections} 

Let $\kk$ be an algebraically closed field, $X$ a smooth irreducible algebraic variety over $\kk$, and $V$ a finite dimensional $\kk$-vector space. 

\begin{definition} An {\bf $n$-parameter (polynomial)
family of flat connections} on $X$ with values in $V$  
is a family of flat connections 
$$
\nabla(\bold s)=d-B(\bold s),
$$ 
$\bold s:=(s_1,...,s_n)$, on the trivial vector bundle $X\times V\to X$, where $B\in \Omega^1(X)\otimes {\rm End}V[\bold s]$
(i.e., $dB-[B,B]=0$, where $d$ is the de Rham differential on $X$ and $[,]$ is the supercommutator in the graded algebra $\Omega^\bullet(X)\otimes {\rm End} V$). Such a family is called a {\bf pencil} if $B$ is linear homogeneous in $\bold s$, i.e. $B=\sum_{j=1}^n s_jB_j$, $B_j\in \Omega^1(X)\otimes {\rm End}V$. 
\end{definition} 

Thus a pencil of flat connections is determined by a collection $$
B_j=\sum_{i=1}^rB_{ij}dx_i\in \Omega^1(X)\otimes {\rm End} V,\ 1\le j\le n
$$   
of 1-forms  
on $X$ with values in ${\rm End} V$
such that 
\begin{equation}\label{commuta} 
dB_j=[B_j,B_k]=0,\ 1\le j,k\le n. 
\end{equation} 

Of course, if $\dim X=1$ then the flatness conditions \eqref{commuta} are vacuous, and $B_j$ can be arbitrary. Yet this will already be an interesting case. 

\begin{remark}\label{affinep} These definitions extend mutatis mutandis to the case when $V$ is a possibly nontrivial vector bundle on $X$. In this case the trivial connection $d$ on the trivial bundle $X\times V\to X$ should be replaced by a fixed flat connection $\nabla(0)$ on the bundle $V$. 
In particular, when $\nabla$ is (affine) linear in $\bold s$, 
we have $\nabla(\bold s)=\nabla(0)-s_1B_1-...-s_nB_n$, 
where $B_i\in \Omega^1(X,{\rm End}V)$. 
In this case we will say that $\nabla$ is an 
{\bf affine pencil of flat connections}. 
Note that if $V$ is trivial, we have 
$\nabla(0)=d-B_0$, so 
$\nabla(\bold s)=d-B$, where $B=B_0+s_1B_1+...+s_nB_n$, 
which is an inhomogeneous (affine) linear function in $\bold s$, which justifies the terminology. Note that the notion of an affine pencil (unlike that of an ordinary pencil) is stable under gauge transformations, so it is the only meaningful notion when the bundle $V$ is not necessarily trivial and there is no canonical reference connection $\nabla(0)$.  

We will, however, mostly restrict ourselves to the case when the bundle $V$ is trivial, since this is so in all the examples we consider. 
\end{remark} 

\subsection{Jumping loci of families of flat connections} \label{julo} 
For a flat connection $\nabla$ on a vector bundle $V$ on $X$, let $\Gamma(\nabla)$ be the space of flat global sections of $\nabla$ 
in $\Gamma(X,V)$. 

Now let $\nabla(\bold s)$, $\bold s\in \Bbb C^n$ be a family of flat connections on the trivial bundle on $X$ with fiber $V$. For $d\in \Bbb N$ let $Q_{d}(\nabla)$ be the set of $\bold s\in \Bbb C^n$ for which $\dim \Gamma(\nabla(\bold s))\ge d$.

\begin{lemma}\label{countunion1} $Q_d(\nabla)$ is a countable union of Zariski closed subsets\footnote{Note that $Q_d(\nabla)$ need not be Zariski closed. For example, if $\nabla(s)=d-\frac{s}{x}dx$ on $\Bbb C^\times$ then 
$Q_1(\nabla)=\Bbb Z$.} of $\Bbb C^n$.
\end{lemma} 

\begin{proof} 
Fix a basis $\lbrace f_i,i\ge 1\rbrace$ of $\Bbb C[X]$, and for $N\le \infty$ let $\Bbb C[X]_N$ be the span of $f_i$ with $i\le N$. For a flat connection $\nabla$ on the trivial bundle on $X$ with fiber $V$, let $\Gamma(\nabla)_N$ be the space of flat sections of $\nabla$ in $\Bbb C[X]_N\otimes V$. 
In particular, $\Gamma(\nabla)_\infty=\Gamma(\nabla)$. 

Now let $\nabla(\bold s)$ be our family. For $d\in \Bbb N$ let $Q_{d,N}(\nabla)$ be the set of $\bold s\in \Bbb C^n$ for which $\dim \Gamma(\nabla(\bold s))_N\ge d$. Since 
$\Gamma(\nabla(\bold s))_N$ is defined inside $\Bbb C[X]_N\otimes V$ 
by a finite system of homogeneous linear equations with coefficients in $\Bbb C[\bold s]$, 
the subsets $Q_{d,N}(\nabla)\subset \Bbb C^n$ are Zariski closed. Thus $Q_d(\nabla)=Q_{d,\infty}(\nabla)=\bigcup_{N\in \Bbb N} Q_{d,N}(\nabla)$ 
is a countable union of Zariski closed subsets. 
\end{proof} 

It is clear that $Q_0(\nabla)=\Bbb C^n$, $Q_{d+1}(\nabla)\subset Q_d(\nabla)$, and $Q_{{\rm dim}V+1}(\nabla)=\emptyset$, so let $d=d(\nabla)$ be the largest integer in $[0,\dim V]$ for which $Q_d(\nabla)=\Bbb C^n$, i.e., the smallest possible value of $\dim \Gamma(\nabla(\bold s))$. The proper subset $J(\nabla):=Q_{d+1}(\nabla)$ of $\Bbb C^n$ is then called the {\bf jumping locus} of $\nabla$. Lemma \ref{countunion1} implies

\begin{corollary} The jumping locus
$J(\nabla)$ is a countable union of proper Zariski closed subsets of $\Bbb C^n$. 
\end{corollary} 

If a countable union of Zariski closed subsets 
of $\Bbb C^n$ is the whole $\Bbb C^n$ then one of them must be the whole $\Bbb C^n$, so 
there exists $N=N(\nabla)\in \Bbb N$ such 
$Q_{d,N}(\nabla)=\Bbb C^n$. Thus for 
any $\bold s\notin J(\nabla)$ we have 
$\Gamma(\nabla(\bold s))_N=\Gamma(\nabla(\bold s))$. 
Hence for any point $\bold x\in X$ the map $J(\nabla)^c\to {\rm Gr}(d,V)$ 
from the complement of $J(\nabla)$ to the Grassmannian of $d$-dimensional subspaces in $V$ given by $\bold s\mapsto \Gamma(\nabla(\bold s))|_{\bold x}$ 
is rational, i.e., is the restriction of a rational map 
$\gamma(\nabla): \Bbb C^n\to {\rm Gr}(d,V)$.
Note that, as any rational map, this map is regular outside 
a Zariski closed subset $T(\nabla)\subset \Bbb C^n$ of codimension $\ge 2$. 

\subsection{Periodic families and pencils of flat connections} 
\begin{definition} A family $\nabla$ is said to be {\bf periodic} if 
there exist {\bf shift operators} 
$$
A_j\in GL(V)(\kk(\bold s)[X]),\ 1\le j\le n
$$
 such that 
$$
\nabla(\bold s+\bold e_j)\circ A_j(\bold s)=A_j(\bold s)\circ \nabla(\bold s),\ 1\le j\le n.
$$
\end{definition} 

\begin{remark} 1. Note that if $\nabla(\bold s)$ has generically trivial endomorphism algebra (i.e., trivial over $\Bbb C(\bold s)$) then $A_j$ are unique up to multiplication by a scalar rational function of $\bold s$. In this case there exist scalar rational functions $\xi_{jk}(\bold s)$ such that 
$$
A_k(\bold s+\bold e_j)A_j(\bold s)=\xi_{jk}(\bold s)A_j(\bold s+\bold e_k)A_k(\bold s).
$$

2. As before, these definitions extend straightforwardly to the case when $V$ is a possibly nontrivial vector bundle on $X$.
\end{remark}

\begin{example}\label{exshift} Let $\kk=\Bbb C$, $X=\Bbb C^\times$, $V=\Bbb C^2$, 
and $\nabla(s)=d-B(s,x)$, where 
$$
B(s,x):=\frac{1}{x}\begin{pmatrix} s& 1\\ 0& 0\end{pmatrix}.
$$
This family is actually an affine pencil, since the dependence on $s$ is affine linear. 
The equation for flat sections of $\nabla(s)$ is 
$$
(\tfrac{d}{dx}-B(s,x))F(x)=0,
$$ 
so it has a
fundamental solution $F_s(x):=\begin{pmatrix} x^s& -\frac{1}{s}\\ 0& 1\end{pmatrix}$.
 Thus $F_s(x)^{-1}=\begin{pmatrix} x^{-s}& \frac{x^{-s}}{s}\\ 0& 1\end{pmatrix}$, so this family is periodic with
$$
A(s,x)=F_{s+1}(x)F_s(x)^{-1}=\begin{pmatrix} x& \frac{x}{s}-\frac{1}{s+1}\\ 0& 1\end{pmatrix}.
$$
\end{example} 

\subsection{Families and pencils with periodic monodromy} 

Let $\nabla=\nabla(\bold s)$ be a family of flat connections on a complex variety $X=X(\Bbb C)$ with values in $V$ parametrized by $\Bbb C^n$. Let $\bold x_0\in X$, and for $\bold s\in \Bbb C^n$ let 
$$
\rho_{\bold s}: \pi_1(X,\bold x_0)\to GL(V)
$$ 
be the monodromy representation of $\nabla(\bold s)$. 

\begin{definition} We say that $\nabla$ has {\bf periodic monodromy} if 
for every $1\le j\le n$, the set $S_j$ of $\bold s$ for which the representation 
$\rho_{\bold s}$ is isomorphic to $\rho_{\bold s+\bold e_j}$ is 
not contained in a countable union of proper analytic subsets of $\Bbb C^n$.
\end{definition} 

In particular, every periodic family has periodic monodromy. 

Let $K:=\Bbb C(\bold q)$, where $\bold q=(q_1,...,q_n)$. 

\begin{theorem}\label{permoncrit} A {\bf pencil} $\nabla$ of flat connections on $X$ with fiber $V$ has periodic monodromy if and only if there exist a nonempty affine open subset $U\subset (\Bbb C^\times)^n$, a finite unramified Galois cover 
$p: \widetilde U\to U$, and 
a Galois-stable representation $\eta: \pi_1(X,\bold x_0)\to GL(V)(\Bbb C[\widetilde U])$ such that the monodromy representation $\rho_{\bold s}$ of $\nabla(\bold s)$ is isomorphic to $\eta|_{e^{2\pi i\bold s}}$ (hence also to $\rho_{\bold s+\bold e_j}$ for all $j$) when $e^{2\pi i\bold s}\in U$. Moreover, in this case $\eta$ is unique up to action of $GL(V)(\overline K)$. 
 \end{theorem} 

\begin{proof} The ``if" direction is obvious, so we only need to prove the ``only if" direction. It is well known that $\pi_1(X,\bold x_0)$ is finitely generated. Thus $\rho_{\bold s}$ may be viewed as a holomorphic family of representations (Example \ref{holfam1}). 

Since $\nabla$ is a pencil (i.e., linear in $\bold s$), for any $g\in \pi_1(X,\bold x_0)$ 
$$
\norm{\rho_{\bold s}(g)}=O(e^{C_g\norm{\bold s}}),\ \norm{\bold s}\to \infty
$$
for some $C_g>0$, i.e., $\rho_{\bold s}(g)$ is exponentially bounded for each $g$. Since $\nabla$ also has periodic monodromy, the result follows from Theorem \ref{holfam}. 
\end{proof} 

Let $\mathcal U:=E^{-1}(U)\subset \Bbb C^n$ and $\pi: \widetilde{\mathcal U}\to \mathcal U$ be the pullback of the cover $p: \widetilde U\to U$ under the map $E: \mathcal U\to U$. 
Denote the corresponding map $\widetilde{\mathcal U}\to \widetilde U$ by $\widetilde E$. 
Thus we have a Cartesian square of unramified analytic covers
\[\begin{tikzcd}
\widetilde{\mathcal U} \arrow{r}{\widetilde E} \arrow[swap]{d}{\pi} &\widetilde U  \arrow{d}{p} \\
\mathcal U \arrow{r}{E} & U
\end{tikzcd}
\]

\begin{proposition}\label{meromo} 
(i) An isomorphism $\phi: \rho_{\pi(\widetilde{\bold s})}\to \eta|_{\widetilde E(\widetilde{\bold s})}$ can be realized by a meromorphic function $\phi(\widetilde{\bold s})$ on $\widetilde{\mathcal U}$ with values in $GL(V)$. 

(ii) Suppose ${\rm End}_{\overline K}\eta=\overline K$, and let 
$U^\circ \subset U$ 
be the (nonempty, Zariski open) set of $\bold a\in U$ for which ${\rm End}\eta|_{\bold a}=\Bbb C$. Let $\phi_*(\widetilde{\bold s})$ be the image of $\phi(\widetilde{\bold s})$ in $PGL(V)$. Then the meromorphic function $\phi_*$ on $\widetilde{\mathcal U}$ is holomorphic on the open subset $\widetilde{\mathcal U}^\circ:=(E\circ \pi)^{-1}(U^\circ)\subset \widetilde{\mathcal U}$. 
\end{proposition}

\begin{proof} 
(i) For $\phi\in {\rm End}V$, being a morphism of representations
$\rho_{\pi(\widetilde{\bold s})}\to \eta|_{\widetilde E(\widetilde{\bold s})}$
is a system of linear equations with coefficients 
in the algebra $R$ generated by some holomorphic functions $h_1,...,h_m$ on $\widetilde {\mathcal U}$. Moreover, for an analytically dense set of $\widetilde{\bold s}\in \widetilde{\mathcal U}$, this system has an invertible solution.  
Thus be Chevalley's elimination of quantifiers, 
there is an invertible solution $\phi$ of this system over ${\rm Frac}(R)$.   

(ii) This follows from the removable singularity theorem. 
\end{proof} 

\begin{remark} Theorem \ref{permoncrit} trivially extends to the case when $\nabla(\bold s)$ is an affine pencil of flat connections on a possibly nontrivial vector bundle $V$ on $X$ (cf. Remark \ref{affinep}). 
\end{remark} 

\subsection{Quantum connections} 

An important class of periodic families of connections is {\bf equivariant quantum connections} $\nabla(\bold s)$ of conical symplectic resolutions $\widetilde Y\to Y$ (see \cite{BMO},\cite{MO} and references therein). These connections are known explicitly in many cases, such as cotangent bundles of partial flag varieties $T^*G/P$, resolutions of Slodowy slices of nilpotent orbits of a simple Lie algebra, quiver varieties, etc. The parameters $s_j$ of such families are the equivariant parameters of the torus $\bold T$ acting on $Y$, the base is $X:=H^2(\widetilde Y,\Bbb C^\times)\setminus D$ where $D$ is a union of subtori, and the fiber is $H^*(\widetilde Y,\Bbb C)$. 
It is known that such families admit {\bf geometric shift operators} $A_j(\bold s)$ which can be defined using enumerative geometry, 
even though they may be difficult to compute explicitly (especially for copies of $\Bbb G_m\subset \bold T$ that dilate the symplectic form). This implies that such families are indeed periodic. 

Also in many examples the quantum connection $\nabla(\bold s)$ is a pencil, i.e., linear in the equivariant parameters $s_i$. Most of the ``trigonometric" examples of Section 5 (trigonometric KZ connections of type $A$, trigonometric Casimir connections, 
trigonometric Dunkl connections) have this property. In general, this happens, for example, when the torus $\bold T$ acting on $X$ has finitely many fixed points (see e.g. \cite{Lee} and references therein). 

\section{Quasi-motivic families and pencils of flat connections} 

\subsection{Quasi-motivic families} 

It is often hard to prove directly that a given family $\nabla$ of flat connections is periodic, as the operators $A_j$ may be difficult to compute. 
However, if $\kk=\Bbb C$ and $\nabla$ has regular singularities, 
there is an approach to checking this property using the monodromy of $\nabla$. This approach is based on the following theorem.  

\begin{theorem}\label{qmper} A family $\nabla$ of flat connections
with regular singularities on a complex variety $X$ has periodic monodromy if and only if it is periodic.  
\end{theorem} 

\begin{proof} Any periodic family has periodic monodromy, so we just need to prove the ``only if" direction. 

Consider the family of flat connections
$\nabla^{(j)}(\bold s):=\nabla(\bold s+\bold e_j)\otimes \nabla(\bold s)^*$. 
Thus $\nabla^{(j)}(\bold s)$ has monodromy representation $\rho_{\bold s+\bold e_j}\otimes \rho_{\bold s}^*$. 

By assumption, the periodicity condition for index $j$ holds for $\bold s$ belonging to some subset $S_j\subset \Bbb C^n$ which is not contained in a countable union of proper Zariski closed subsets. For a fixed $j$ and $\bold s\in S_j$ fix an isomorphism $\psi_{\bold s}: \rho_{\bold s}\cong \rho_{\bold s+\bold e_j}$. There exists a flat holomorphic section $A_j(\bold s)$ of $\nabla^{(j)}(\bold s)$ such that $A_j(\bold s,\bold x_0)=\psi_{\bold s}$; in particular, $A_j(\bold s,\bold x)\in GL(V)$ for all $\bold x\in X$. 

Fix a smooth compactification $\overline X$ of $X$ such that 
$\overline X\setminus X$ 
is a normal crossing divisor (it exists by 
Hironaka's theorem). Since $\nabla^{(j)}(\bold s)$ 
has regular singularities, the functions $A_j(\bold s,\bold x)$ have at worst power growth near $\overline X\setminus X$. Hence they are meromorphic functions on $\overline X$. 
Thus $A_j(\bold s,\bold x)$ are rational functions, i.e., $A_j(\bold s)\in 
GL(V)(\Bbb C[X])$. 

Fix a basis $\lbrace f_i,i\in \Bbb Z_{\ge 1}\rbrace$ of $\Bbb C[X]$. 
For $M\in \Bbb N$ let $S_j(M)\subset S_j$ be the set 
of $\bold s\in S$ for which 
$$
A_j(\bold s,\bold x)=\sum_{l=1}^M A_{jl}(\bold s)f_l(\bold x).
$$
for suitable matrices $A_{jl}(\bold s)$. It is clear that $S_j(M)\subset S_j(M+1)$ 
and $\cup_M S_j(M)=S_j$. Since $S_j$ is not contained in a countable union of proper Zariski closed subsets, there exists $M$ for which $S_j(M)$ 
is Zariski dense in $\Bbb C^n$. The property that $A_j(\bold s,\bold x)$ is a flat section of $\nabla^{(j)}(\bold s)$ is then a system of homogeneous linear equations on the matrices $A_{jl}(\bold s)$, $l=1,...,M$, over $\Bbb C(\bold s)$. By Chevalley's elimination of quantifiers, since this system has an invertible solution for $\bold s\in S_j(M)$, it must have one over $\Bbb C(\bold s)$, where $\bold s$ is a collection of variables, as claimed. 
\end{proof} 

\begin{remark} If $\nabla$ has irregular singularities, then it can have periodic monodromy without being periodic, for example, $\nabla(s)=d-s$ on $X=\Bbb A^1$ (the monodromy is trivial for all $s$ since $X$ is simply connected). We expect, however,  that Theorem \ref{qmper} can be extended to the irregular case (at least when $\dim X=1$) if 
the condition of periodic monodromy is replaced by ``periodic monodromy and Stokes data".  
We hope that this can be done using the techniques of \cite{Ka}. 
\end{remark} 

Theorem \ref{qmper} motivates the following definition. 

\begin{definition} We say that a family $\nabla$ of flat connections over an algebraically closed field $\kk$ of characteristic zero is {\bf quasi-motivic} if 

(i) $\nabla(\bold s)$ has regular singularities for all $\bold s\in \kk^n$, and 

(ii) $\nabla$ is periodic. 
\end{definition} 

By Theorem \ref{qmper}, if $\kk=\Bbb C$ then $\nabla$ is quasi-motivic if and only if it has regular singularities and periodic monodromy. 

\subsection{Jumping loci of quasi-motivic pencils} 
In the next few subsections we investigate various kinds of singularities of quasi-motivic pencils. Since we are using the Riemann-Hilbert correspondence, we will work over $\kk=\Bbb C$, but as our statements are purely algebraic, they hold, by standard abstract nonsense, over any algebraically closed field $\kk$ of characteristic zero. 

Let $\nabla$ be a quasi-motivic pencil. Let $Z:=U^c\subset (\Bbb C^\times)^n$ be the complement of the set $U$ from Theorem \ref{permoncrit}, a divisor in $(\Bbb C^\times)^n$. Let $Z_{\rm ht}\subset Z$ be the union of the components of $Z$ which are (affine) hypertori. Then $\mathcal Z:=E^{-1}(Z_{\rm ht})\subset \Bbb C^n$ is an arrangement of 
translates of finitely many hyperplanes defined over $\Bbb Q$, with equations
$$
a_{i1}s_1+...+a_{in}s_n=c_i+N,\ 1\le i\le p,
$$ 
where $a_{ij}\in \Bbb Z$ (coprime for each $i$), $N\in \Bbb Z$, $c_i\in \Bbb C$. 

\begin{theorem} \label{permoncrit1} Every codimension $1$ irreducible Zariski closed subset $H$ of the jumping locus $J(\nabla)$ is a hyperplane contained in $\mathcal Z$. 
\end{theorem}

\begin{proof} Let $\rho_{\bold s}$ be the monodromy representation of $\pi_1(X,\bold x_0)$ defined by $\nabla(\bold s)$. Since $\nabla(\bold s)$ has regular singularities, $\Gamma(\nabla(\bold s))\cong {\rm Hom}_{\pi_1(X,\bold x_0)}(\Bbb C,\rho_{\bold s})$, thus 
$J(\nabla)$ is the set of $\bold s$ for which ${\rm dim} {\rm Hom}_{\pi_1(X,\bold x_0)}(\Bbb C,\rho_{\bold s})> d(\nabla)$. Since $\nabla$ has periodic monodromy, by Theorem \ref{permoncrit}, $E(H)$ is contained in $Z$. So Theorem \ref{Zil} implies 
that $E(H)$
is contained in $Z_{\rm ht}$, as desired. 
\end{proof} 

\begin{example} Since $\Gamma(\nabla\otimes \nabla^*)={\rm End}(\nabla)$ is the endomorphism algebra of $\nabla$, Theorem \ref{permoncrit1}
applies to the jumping locus $J(\nabla\otimes \nabla^*)$ for the dimension of ${\rm End}(\nabla)$.
Thus every codimension $1$ irreducible Zariski closed subset $H$ of the set of $\bold s\in \Bbb C^n$ where ${\rm End}(\nabla(\bold s))$ exceeds its generic value is a 
hyperplane contained in $\mathcal Z$. 
\end{example} 

\subsection{Poles of shift operators for quasi-motivic pencils} 
Let us choose the shift operators $A_j$ for $\nabla$. Let $\mathcal B_j$ be the set of exceptional values of $\bold s$ where $\nabla(\bold s)$ is not isomorphic to $\nabla(\bold s+\bold e_j)$.  
It is clear that $\mathcal B_j$ is a contained in the pole 
divisor $\mathcal Z_j\subset \Bbb C^n$ of the projection $A_{j*}$ of $A_j$ to $PGL(V)$. Thus every codimension $1$ irreducible component of the Zariski closure 
$\overline{\mathcal B}_j$ is a component of $\mathcal Z_j$. 

\begin{theorem}\label{badval} (i) All codimension $1$ irreducible components of $\overline{\mathcal B}_j$ are contained in $\mathcal Z$.

(ii) If $\nabla$ has generically trivial endomorphism algebra (i.e., trivial over $\Bbb C(\bold s)$) then $\mathcal Z_j$ are contained in $\mathcal Z$.
\end{theorem} 

\begin{proof} (i) Since $\nabla$ has regular singularities, by the Riemann-Hilbert correspondence (\cite{De1})
we have $\mathcal B_j\subset E^{-1}(Z)$. 
Hence 
$\overline{\mathcal B_j}\subset  E^{-1}(Z)$, since 
$E^{-1}(Z)$ is closed in the Euclidean topology, and for a constructible set 
the Zariski and Euclidean closures are the same. 
Now, by Theorem \ref{Zil}, every component of $\mathcal Z_j$ that is contained in 
$E^{-1}(Z)$ is, in fact, contained in $E^{-1}(Z_{\rm ht})$. 
This implies (i). 

(ii) Let $\widetilde {\mathcal{U}}$ be as in Proposition \ref{meromo}. 
Since $\nabla$ has generically trivial endomorphism algebra, by Subsection \ref{julo}
this is so outside a countable union of hypersurfaces, so 
${\rm End}_{\overline K}\eta=\overline{K}$. Thus, in view of Theorem \ref{permoncrit1}, 
by Theorem \ref{permoncrit} and Proposition \ref{meromo}(ii) we have an isomorphism $\widehat A_j(\widetilde{\bold s}): \rho_{\bold s}\to \rho_{\bold s+\bold e_i}$ which depends holomorphically on $\widetilde {\bold s}\in (E\circ \pi)^{-1}(U)$, where $\bold s=\pi(\widetilde {\bold s})$. 
Let $\widehat A_{j*}(\widetilde{\bold s})$ be the image of $\widehat A_j(\widetilde{\bold s})$ in $PGL(V)$. Since the endomorphism algebra of $\nabla(\bold s)$ is trivial when $\widetilde {\bold s}\in \widetilde{\mathcal U}$, we have $\widehat A_{j*}(\widetilde{\bold s})=A_{j*}(\bold s)$, $\widetilde{\bold s}\in \widetilde{\mathcal U}$. It follows that $A_{j*}$ is regular on $E^{-1}(U)$, so $\mathcal Z_j\subset E^{-1}(Z)$. Thus by Theorem \ref{Zil}, $\mathcal Z_j\subset E^{-1}(Z_{\rm ht})=\mathcal Z$. 
\end{proof} 

\subsection{Endomorphism algebras of quasi-motivic pencils}  
For a finite dimensional $GL(V)$-module $W$, let $\nabla_W$ be the flat connection 
on $X$ with fiber $W$ associated to $\nabla$. 
It is clear that if $\nabla$ has regular singularities 
then so does $\nabla_W$, and if 
$\nabla(\bold s)$ has periodic monodromy, is periodic, or is quasi-motivic, then so is $\nabla_W(\bold s)$

Let $d_W=d_W(\nabla):=d(\nabla_W)$. 
Let $Y\subset {\rm Gr}(d_W,W)$ be the Zariski closure of the image of the rational map 
$\gamma=\gamma(\nabla_W)$ defined in Subsection \ref{julo}, an irreducible closed subvariety.

\begin{theorem}\label{densor} The action of $GL(V)$ on $Y$ has a (unique) dense orbit $Y^\circ$. 
\end{theorem} 

\begin{proof} By Rosenlicht's theorem in invariant theory (\cite{PV}, 2.3), rational invariants separate generic orbits. Thus it suffices to show that $\Bbb C(Y)^{GL(V)}=\Bbb C$. Let $F$ be a $GL(V)$-invariant rational function on $Y$. Then $F\circ \gamma$ is a rational function of $\bold s\in \Bbb C^n$. On the other hand, since $\nabla$ has regular singularities, we have $\gamma(\bold s)=W^{\pi_1(X,\bold x_0)}$, where $\pi_1(X,\bold x_0)$ 
acts on $W$ through the monodromy representation $\rho_{\bold s}$. Thus by Theorem \ref{permoncrit}, $F\circ \gamma$ is also rational in $e^{2\pi i\bold s}$. This implies the statement, since a function simultaneously rational in $\bold s$ and $e^{2\pi i\bold s}$ must be constant. 
\end{proof} 

Let $\mathcal U(\nabla,W)$ be the Zariski open set of $\bold s\in \Bbb C^n$ such that 
$\gamma(\bold s)$ is defined and belongs to $Y^\circ$. Thus for any two 
elements $\bold s,\bold s'\in \mathcal U(\nabla,W)$, 
the corresponding subspaces $\gamma(\bold s)$ and 
$\gamma(\bold s')$ are conjugate under the action of $GL(V)$. 

\begin{corollary}\label{coo1} Every codimension 1 component $H$ of the complement 
$\mathcal U(\nabla,W)^c$ is a hyperplane contained in $\mathcal Z$. 
\end{corollary} 

\begin{proof}
Since the membership 
of $\bold s$ in $\mathcal U(\nabla,W)$ depends only on $\rho_{\bold s}$, 
$E(H)$ is contained in $Z$. So by Theorem \ref{Zil} it is contained in $Z_{\rm ht}$ and the statement follows. 
\end{proof} 

Now consider the special case $W=V\otimes V^*$. In this case 
$\gamma(\bold s)$, when defined, is the {\bf reduced endomorphism algebra} 
${\rm End}_{\rm red}(\nabla(\bold s))$. If $\bold s\notin J(\nabla)$, it coincides 
with the usual endomorphism algebra ${\rm End}(\nabla(\bold s))$, 
otherwise it is a subalgebra in ${\rm End}(\nabla(\bold s))$ obtained as the limit 
of the endomorphism algebra from the generic locus. 

\begin{corollary}\label{coo2} 
For values of $\bold s$ outside finitely many hyperplanes contained in $\mathcal Z$ and a Zariski closed subset of $\Bbb C^n$ of codimension $\ge 2$, the 
reduced endomorphism algebra ${\rm End}_{\rm red}(\nabla(\bold s))$ together with its action on $V$ does not depend on $\bold s$ up to an isomorphism. 
\end{corollary} 

\begin{proof} It suffices to apply Corollary \ref{coo1} to $W=V\otimes V^*$.  
\end{proof} 

Let us regard ${\rm Aut}_{\rm red}(\nabla(\bold s))={\rm End}_{\rm red}(\nabla(\bold s))^\times$ as a group scheme over $\mathcal U(\nabla,V\otimes V^*)$. 
We can then project the shift operators $A_j(\bold s)$ to rational sections $A_{j\bullet}(\bold s)$ of the bundle over $\mathcal U(\nabla,V\otimes V^*)$ 
with fiber $GL(V)/{\rm Aut}_{\rm red}(\nabla(\bold s))$. It is clear that 
$A_{j\bullet}(\bold s)$ does not depend on the choice of $A_j$. 
Let $\mathcal Z_{j\bullet}$ be the pole divisor of $A_{j\bullet}$. 
We obtain the following generalization of Theorem \ref{badval}(ii). 

\begin{corollary}\label{coo3} We have $\mathcal Z_{j\bullet}\subset \mathcal Z$. 
\end{corollary} 
 
\begin{proof} Using the same argument as in the proof of Theorem \ref{badval}(ii), one establishes that $E(\mathcal Z_{j\bullet})\subset Z$. Thus by Theorem \ref{Zil}, 
$E(\mathcal Z_{j\bullet})\subset Z_{\rm ht}$, so the result follows. 
\end{proof}  

\subsection{Semisimplicity loci of quasi-motivic pencils}  
Another similar result is the following theorem, which is proved analogously to the above results by applying 
Theorem \ref{Zil}. Let ${\rm Semis}(\nabla)\subset \Bbb C^n$
be the subset of points $\bold s$ where $\nabla(\bold s)$ is semisimple. 
This means that in some basis $\nabla(\bold s)$ is block-diagonal with blocks of some sizes $N_1,...,N_p$ (where $N_1+...+N_p=\dim V$), but it is not block-triangular with $p+1$ blocks in any basis. So similarly to Subsection \ref{julo}, one shows that the complement
${\rm Semis}(\nabla)^c$ is a countable union of locally closed subsets of $\Bbb C^n$. 

\begin{theorem} \label{permoncrit3} Suppose that $\nabla$ is generically semisimple (i.e., semisimple over 
$\Bbb C(\bold s)$). Then the closure of every codimension $1$ irreducible locally closed subset of ${\rm Semis}(\nabla)^c$ is a hyperplane contained in $E^{-1}(T)$, where $T$ is a union of affine hypertori in $(\Bbb C^\times)^n$. 
\end{theorem} 

\begin{remark} We will see in Section 5 that many natural examples of quasi-motivic pencils 
are generically semisimple.
\end{remark} 

\begin{remark} In general, the kind of result similar to Theorem \ref{permoncrit3} holds (with the same proof) for any property of $\nabla$ that is algebraic in both de Rham and Betti realizations and preserved by the Riemann-Hilbert correspondence. 
\end{remark} 

\subsection{Motivic families} 
An important class of quasi-motivic families is {\bf motivic families}. 
To define them, let us introduce some notation. Let $Y$ be a smooth complex variety and 
 $\Phi_1,...,\Phi_n$
 non-vanishing regular functions on $Y$.
 Denote by  $\mathcal L(\Phi_1,...,\Phi_n,\bold s)$
 the de Rham local system on $Y$ 
generated by  the multivalued function $\prod_{j=1}^n \Phi_j^{s_j}$.

Let $\nabla(\bold s)$ be a family of flat connections on an irreducible affine complex variety $X$.

\begin{definition} We say that $\nabla$ is {\bf motivic} if there exists an irreducible variety $Y$, a smooth morphism $\pi: Y\to X$, and non-vanishing regular functions $\Phi_1,...,\Phi_n$ on $Y$ such that on some dense open set $X^\circ \subset X$, the connection $\nabla(\bold s)$ is isomorphic  to the Gauss-Manin connection 
$\pi_*\mathcal L(\Phi_1,..,\Phi_n,\bold s)$ 
on $H^i(\pi^{-1}(\bold x),\mathcal L(\Phi_1,..,\Phi_n,\bold s))$ for Zariski generic $\bold s$ and
some $i$.
\end{definition} 

For example, if $\pi$ is affine, then in informal terms this means that the local flat sections $F(\bold x)$ of $\nabla(\bold s)$ can be obtained by integrating the function $\prod_{j=1}^n \Phi_j^{s_j}(\bold y)$ against a regular differential form $\omega$ on $Y$  over a twisted singular cycle $C_{\bold x}$ in the fiber $Y_{\bold x}=\pi^{-1}(\bold x)$ of (real) dimension $\dim Y_{\bold x}-i$, varying continuously with $\bold x$: 
$$
F(\bold x)=\int_{C_{\bold x}} \prod_{j=1}^n \Phi_j^{s_j}\omega. 
$$ 

Let $K=\Bbb Q(\bold q)$. We say that a quasi-motivic family $\nabla$ on a complex variety $X$ has {\bf algebraic monodromy} if there exists a representation $\eta: \pi_1(X,\bold x_0)\to GL_N(\overline K)$ such that 
the monodromy representation $\rho_{\bold s}$ of $\nabla(\bold s)$ is isomorphic to $\eta|_{\bold q=e^{2\pi i\bold s}}$
for Zariski generic $e^{2\pi i\bold s}\in (\Bbb C^\times)^n$, 
and 
{\bf rational monodromy} if $\eta$ can be chosen to land in $GL_N(K)$. 
In light of Theorem \ref{permoncrit}, algebraic monodromy is just the requirement that 
the numerical coefficients of $\eta$, which are a priori complex numbers, actually belong to $\overline{\Bbb Q}$. 

\begin{proposition}\label{motqmot} Every motivic family $\nabla$ is quasi-motivic and has rational monodromy.
\end{proposition} 

\begin{proof} This follows since a Gauss-Manin connection always has regular singularities and its monodromy can be computed in the Betti realization. 
\end{proof} 

Hence every motivic family is periodic by Theorem \ref{qmper}. 
In fact, this is also easy to see directly, as multiplication by $\Phi_j$ 
defines an isomorphism 
$$
a_k(\bold s): \mathcal L(\Phi_1,...,\Phi_n,\bold s)\cong 
\mathcal L(\Phi_1,...,\Phi_n,\bold s+\bold e_j)
$$ 
which gives rise to an isomorphism 
$A_k(\bold s)=\pi_*(a_k(\bold s))$ of the corresponding direct images. 

\begin{remark} Our results about periodic, quasi-motivic and motivic families of flat connections extend in a straightforward way to {\bf rational} families, i.e. those with coefficients in $\Bbb C(\bold s)$ rather than in $\Bbb C[\bold s]$. 
\end{remark} 

\begin{remark} We will see below that many quasi-motivic families are in fact motivic, or, more generally, can be realized as subquotients of motivic families (at least for generic $\bold s$). In fact, we do not know an example of a generically irreducible quasi-motivic family which does not have this property (possibly after shifting $\bold s$), although there are many examples of such families with no known motivic realization.  
\end{remark} 

\subsection{A speculative digression: quasi-geometric  connections} 

\subsubsection{Quasi-geometric connections} 

\begin{definition} Let $X$ be a smooth irreducible complex 
algebraic variety defined over $\overline{\Bbb Q}$. 
Let $\nabla$ be a flat connection on a vector bundle $V$ on $X$ with regular singularities. We say that $\nabla$ is {\bf quasi-geometric} if the monodromy representation of the corresponding flat holomorphic connection $\nabla_{\Bbb C}$ on $X$ viewed as a complex manifold is also defined over 
$\overline{\Bbb Q}$. 
\end{definition} 

In other words, a quasi-geometric connection is 
a flat connection with regular singularities
 defined over $\overline{\Bbb Q}$ both in the de Rham and the Betti realization. 

Note that if $X^\circ\subset X$ is a non-empty Zariski open subset and $\bold x_0\in X^\circ$ then the map 
$\pi_1(X^\circ,\bold x_0)\to \pi_1(X,\bold x_0)$ is surjective, so $\nabla$ is quasi-geometric on $X$ if and only if it is quasi-geometric on $X^\circ$. 

\begin{example} 1. A geometric connection (Subsection \ref{geoloc}) is quasi-geometric. 

2. The category of quasi-geometric connections is invariant under pullbacks, external products (hence tensor products), and duals. Thus for every $X$, we have a $\overline{\Bbb Q}$-linear Tannakian category 
$\mathcal Q\mathcal G(X)$ of quasi-geometric connections on $X$. 

3. If $\pi: Y\to X$ is a smooth morphism and $\nabla$ is a quasi-geometric connection on $Y$ then for all $i$ the $i$-th cohomology of $\pi_*\nabla$ is quasi-geometric on its smooth locus $X^\circ\subset X$, since the direct image can be computed over $\overline{\Bbb Q}$ both in the de Rham and the Betti realization. 

4. A simple composition factor of a quasi-geometric connection is quasi-geometric. 

5. An irreducible {\bf rigid} connection (i.e., one determined up to finitely many choices by conjugacy classes of its local monodromies) is quasi-geometric, since it is deformation-theoretically rigid 
both in the de Rham and the Betti realization. Note that for $X=\Bbb P^1\setminus S$ where $S\subset \overline{\Bbb Q}$ is a finite set, N. Katz showed 
that rigid connections are geometric by 
using the method of middle convolution (\cite{K}). 
\end{example}

The relevance of the notion of a quasi-geometric connection
to this paper is exhibited by the following proposition, whose proof is straightforward. 

\begin{proposition}\label{qgeo} If $\nabla$ is a quasi-motivic family of flat connections on $X$  defined over $\overline{\Bbb Q}$ with algebraic monodromy, then for $\bold s\in \Bbb Q^n$, $\nabla(\bold s)$ is quasi-geometric. 
\end{proposition} 

In particular, this means that the examples of Section 5 are quasi-geometric (we will see that they have algebraic monodromy). 

An important property of quasi-geometric connections is 
quasiunipotent monodromy at infinity. 

\begin{proposition}\label{quasiun} Let $\overline X$ be a smooth compactification of $X$ such that $\overline X\setminus X$ is a normal crossing divisor. If $\nabla$ is a quasi-geometric connection on $X$, then its monodromy around every component $D$ of $\overline X\setminus X$ is quasiunipotent. 
\end{proposition} 

\begin{proof} Let us extend the underlying vector bundle 
of $\nabla$ over $D$ so that $\nabla$ has a first order pole there (this is possible since $\nabla$ has regular singularities, \cite{De1}). Then the eigenvalues of the monodromy of $\nabla$ around $D$ are of the form $e^{2\pi is}$, where $s$ runs through the eigenvalues of the residue of $\nabla$ at $D$. Thus both $s$ and $e^{2\pi is}$ are algebraic numbers. 
Hence it follows from the theorem of Gelfond and Schneider in transcendental number theory (\cite{Ge}) that $s\in \Bbb Q$, i.e., $e^{2\pi is}$ is a root of unity, as desired. 
\end{proof} 

\begin{remark} Since every geometric connection is quasi-geometric, Proposition \ref{quasiun} 
is a generalization of the Local Monodromy Theorem in Hodge theory: monodromy at infinity of a Gauss-Manin connection is 
quasiunipotent. In fact, our proof of Proposition \ref{quasiun}
is the same as Brieskorn's proof of this theorem given at the end of \cite{De1}. 
\end{remark} 

\begin{remark} There are many examples of irreducible quasi-geometric connections without a known geometric construction, even on $X=\Bbb P^1\setminus S$. E.g., we can take the Dunkl connection \ref{ratdun} for exceptional Coxeter groups at a rational parameter $c$, restricted to a generic line in $\mathfrak h$ (not passing through $0$). However, we do not know an example of an irreducible quasi-geometric connection which is provably non-geometric. For instance, Proposition \ref{quasiun} implies that a 1-dimensional connection on $\Bbb P^1\setminus S$ which is quasi-geometric is actually geometric.  
\end{remark} 

Let $\nabla$ be a connection on $X:=\Bbb P^1\setminus S$ with regular singularities defined over $\overline{\Bbb Q}$ whose monodromies around points of $S$ are quasiunipotent (i.e., the eigenvalues of the residues are rational). In this case it is usually difficult to show that $\nabla$ is not quasi-geometric, because this entails establishing transcendence of certain numbers (such as traces of monodromy operators).
Yet we expect that the quasi-geometric property is quite rare. Namely, we make the following conjecture. 

\begin{conjecture}\label{fii} Fix a quasiunipotent Riemann symbol $P$ for $S$ (i.e., the list of eigenvalues 
of monodromies around points of $S$). Then there are finitely many quasi-geometric connections 
on $\Bbb P^1\setminus S$ with Riemann symbol $P$. 
\end{conjecture} 

In particular, this would mean that already for rank $2$ and $|S|=4$, most connections with a given $P$ are 
not quasi-geometric, since such connections are parametrized by two parameters in $\overline{\Bbb Q}$. 

\begin{remark} Conjecture \ref{fii} holds for geometric connections by a result of Deligne \cite{De}, which provides 
motivation and supporting evidence for this conjecture. 
\end{remark} 
 
\subsubsection{Quasi-geometric connections and braided fusion categories} 

In spite of likely being rare among general regular connections 
with quasiunipotent local monodromies, quasi-geometric connections are ubiquitous 
in different areas of mathematics. In particular, 
we would like to propose the following conjecture, stating that 
connections arising from braided fusion categories are always quasi-geometric. 

Let $\mathcal C$ be a {\bf braided fusion category} over $\Bbb C$ (\cite{EGNO}, Definitions 4.1.1, 8.1.1). 
By the Ocneanu rigidity theorem, $\mathcal C$ is defined over $\overline{\Bbb Q}$ (\cite{EGNO}, 
Corollary 9.1.8). Thus for any objects $X_1,...,X_n,Y\in \mathcal C$, the space ${\rm Hom}(Y,X_1\otimes...\otimes X_n)$ carries an action of the pure braid group $PB_n$, which defines a Betti local system ${\mathcal L}_{X_1,...,X_n,Y}$
on the configuration space $X_n:=\Bbb C^n\setminus{\rm diagonals}$ defined over $\overline {\Bbb Q}$. 
By the Riemann-Hilbert correspondence (\cite{De1}), there is a unique up to isomorphism flat connection $\nabla_{X_1,...,X_n,Y}$ on $X_n$ with regular singularities whose monodromy is ${\mathcal L}_{X_1,...,X_n,Y}$. 

\begin{conjecture}\label{quasigeo} The connection 
$\nabla_{X_1,...,X_n,Y}$ is defined over $\overline{\Bbb Q}$. In other words, 
it is a quasi-geometric connection. 
\end{conjecture} 

If $\mathcal C$ is a modular category (\cite{EGNO}, Subsection 8.14) 
then Conjecture \ref{quasigeo} can be strengthened 
to say that the flat connections with regular singularities on the moduli 
space of curves of any genus $g$ with $n$ punctures attached to mapping class group representations 
on spaces of conformal blocks of the category $\mathcal C$ are defined over $\overline{\Bbb Q}$ 
(equivalently, are quasi-geometric). In genus $0$ this is actually equivalent to Conjecture 
\ref{quasigeo}, since the category $\mathcal C$ can be assumed modular by pivotalization (\cite{EGNO}, p.180) and taking the Drinfeld center (\cite{EGNO}, Subsection 7.13).

A possible approach to proof of Conjecture \ref{quasigeo} would be to 
establish a de Rham analog of Ocneanu rigidity, using the de Rham 
description of the notion of a modular fusion category and modular functor 
in terms of the Teichm\"uller tower, due to Deligne and Beilinson-Feigin-Mazur (see \cite{BK}, Chapter 6 and references therein). 

Some evidence for Conjecture \ref{quasigeo} is provided by the following proposition.

\begin{proposition}\label{voa} If $\mathcal C$ is the category of representations of a strongly rational vertex algebra $\mathcal V$ (in the sense of \cite{M}, p.1) defined over $\overline {\Bbb Q}$ then Conjecture \ref{quasigeo} holds for $\mathcal C$ (in the generalized form for any genus).
\end{proposition} 

\begin{proof} By Huang's theorem (\cite{H}), $\mathcal C$ is a modular category. 
Since $\mathcal V$ is defined over $\overline {\Bbb Q}$, so are the KZ  type connections on genus $0$ conformal blocks for $\mathcal V$. 
On the other hand, the monodromy of these connections is also defined
over $\overline{\Bbb Q}$ by Ocneanu rigidity (\cite{EGNO}, 
Corollary 9.1.8), as already mentioned above. This proves the proposition. 
\end{proof} 

Note that it is conjectured (see e.g. \cite{GJ}, p.2) that any modular category $\mathcal C$ is the category of representations of a (strongly) rational vertex algebra $\mathcal V$. If so, and $\mathcal V$ is defined over $\overline {\Bbb Q}$, then Conjecture \ref{quasigeo} would follow from Proposition \ref{voa}. 

We do not know if any rational vertex algebra is defined over $\overline{\Bbb Q}$ (although this seems to be true in all known examples). But we hope that this can be proved (maybe under mild assumptions) if $\mathcal V$ is deformation-theoretically rigid (has no non-trivial formal deformations), which could presumably be deduced from vanishing of a suitable second cohomology group. We do not know, however, how to prove such rigidity. 

Also, in all examples of vertex algebras we know, 
the KZ  type connections on genus $0$ conformal blocks
are Gauss-Manin connections (i.e., their local flat sections have integral representations). 
For instance, if $\mathcal V$ is the simple vertex algebra attached to the affine Kac-Moody algebra 
$\widehat{\mathfrak g}$ at level $k$ then these connections are literally the KZ  connections, whose Gauss-Manin realization (i.e., integral form of solutions) is given in \cite{SV}. 
In such cases, one may say that the quasi-geometric nature of the corresponding flat connections 
is explained by the fact that they are Gauss-Manin connections (i.e., geometric). 

Finally, we note that there are concrete examples of modular fusion categories (coming from subfactors)
for which the flat KZ-type connections attached via the Riemann-Hilbert correspondence to their braid group representations are not known explicitly (the problem of computing a regular connection with given monodromy is, in general, highly transcendental, even on $X=\Bbb P^1\setminus S$). In particular, they are not known to be Gauss-Manin connections, not even quasi-geometric ones, so Conjecture \ref{quasigeo} is open for them. The simplest such example is perhaps the Drinfeld center of the Haagerup fusion category, see \cite{GI}. This category is not known to be the representation category of a rational vertex algebra. 

\begin{remark} If a braided fusion category $\mathcal C$ is unitary (\cite{Tu}) then the representations of the pure braid group $PB_n$ on its multiplicity spaces are also unitary, thus semisimple. Hence they are also semisimple if $\mathcal C$ is Galois conjugate to a unitary category. We do not know, however, whether representations of $PB_n$ coming from braided fusion categories must be semisimple in general. This is true in the abelian case $n=2$ by the Anderson-Moore-Vafa theorem (\cite{EGNO}, 8.18), 
but nothing is known for $n>2$.
\end{remark} 

\subsection{Pseudo-pencils} Even though the pencil condition on a family 
 of flat connections $\nabla(\bold s)$ (i.e., linearity in $\bold s$) is satisfied in many interesting examples listed in the next section, it is still quite restrictive and unfortunately rules out many other important examples, for instance most motivic families. On the other hand, we have used this condition only in one place - to argue that the monodromy representation of $\nabla(\bold s)$ is exponentially bounded. It turns out that the pencil condition can be relaxed to include many more examples while retaining this property. Namely, this is accomplished by the following notion of a {\bf pseudo-pencil}, which we first define in the 1-parameter case. 
 
\begin{definition} Let $\nabla(s)$ be a 1-parameter polynomial family of flat connections on a smooth irreducible variety $X$. We say that $\nabla(s)$ is a {\bf pseudo-pencil} if after a gauge transformation over $\bold k(s)(X)$, $\nabla(s)$ takes the form $d-sB(s)$ where $B$ is regular at $s=\infty$. 
\end{definition} 

It is clear that any pencil of flat connections is a pseudo-pencil. On the other hand, pseudo-pencils are much more general. For instance, using the results of Mochizuki (\cite{Mo}), one can prove the following proposition (see e.g. \cite{Sa}, Theorem 2 and Example 3): 

\begin{proposition}\footnote{We thank V. Vologodsky for explanations regarding this proposition} Let $\pi: Y\to X$ be a smooth morphism between smooth complex varieties of relative dimension $m$ and $\Phi$ a nonvanishing regular function on $Y$ such that the map $(\pi,\Phi): Y\to X\times \Bbb G_m$ is proper. Suppose that the critical points of $\Phi$ on a generic fiber $\pi^{-1}(\bold x)$ are isolated. Then the motivic family $H^m(\pi_*\mathcal L(\Phi,s))$ (over some open subset of $X$) is a pseudo-pencil. 
\end{proposition} 

There are also other similar situations where one can conclude that $H^m(\pi_*\mathcal L(\Phi,s))$ is a pseudopencil, 
see e.g. \cite{Sa}, Example 4. 

Yet it is clear that the monodromy of any pseudo-pencil (for $\bold k=\Bbb C$) is exponentially bounded (indeed, we may compute the monodromy in the gauge where $\nabla(s)\sim d-sB(s)$, where it is obviously exponentially bounded, but this condition does not depend on the gauge). 

This can be generalized to the multiparameter case as follows. 

\begin{definition} Let $\nabla(\bold s)$, $\bold s=(s_1,...,s_n)$ be an $n$-parameter polynomial family of flat connections on $X$. We say that $\nabla(\bold s)$ is a {\bf pseudo-pencil} if it is so with respect to each $s_i$ when the other coordinates $s_j$, $j\ne i$ are fixed. 
\end{definition} 

Again, it is easy to see that any pencil is a pseudo-pencil, but pseudo-pencils are much more general. Yet the above results can be generalized to them. 

To make this generalization, note that, as shown by the proof of Lemma \ref{complan}, the exponential boundedness assumption in this Lemma and hence in Theorem \ref{holfam} can be relaxed to the assumption of {\bf coordinatewise exponential boundedness}: we say that $f: \Bbb C^n\to V$ is coordinatewise exponentially bounded for all $i$, $f(\bold s)$ is exponentially bounded with respect to $s_i$ when $s_j$ for $j\ne i$ are fixed. 

It remains to note that the monodromy of any pseudo-pencil is coordinatewise exponentially bounded, which enables the desired generalizations. 

Finally, we remark that this discussion extends straightforwardly to the case 
when $\nabla(\bold s)$ is a rational (rather than polynomial) function of $\bold s$.  

\section{Examples of periodic pencils} 

In this section we will review numerous examples of 
pencils $\nabla(\bold s)$ of flat connections arising in representation theory and mathematical physics. All of them will be shown to be periodic, possibly after rescaling $\bold s$ by an integer factor. In most cases, this is done by showing that these pencils are quasi-motivic. 

\subsection{KZ  connections}

\subsubsection{KZ  connections for finite dimensional simple Lie algebras}\label{kz1}
Let $\g$ be a finite dimensional simple Lie algebra over $\Bbb C$ with triangular decomposition $\g=\mathfrak n_-\oplus \mathfrak h\oplus \mathfrak n_+$ and root system $R\subset \mathfrak h^*$, and let $\lbrace h_i\rbrace$ be the Chevalley basis of $\mathfrak h$. Let $\Omega\in (S^2\g)^\g$ be the Casimir tensor of $\g$ corresponding to the invariant inner product $(,)$ on $\g^*$ under which short roots of $\g$ have squared length $2$:
$$
\Omega=\sum_{i=1}^{{\rm rank}(\g)} h_i\otimes h_i^*+\sum_{\alpha\in R}e_\alpha\otimes e_{-\alpha},  
$$
where $h_i^*$ is the dual basis of $\mathfrak h$ under $(,)$ and $e_\alpha$ are suitably normalized root elements of $\g$. 
Let $V_1,...,V_r$ be representations of  $\g$ from category $\mathcal O$ with rational weights. Then for any rational weight $\mu\in \mathfrak h^*$ the weight subspace 
$(V_1\otimes...\otimes V_r)[\mu]$ is finite dimensional. The KZ  connection 
on $X=\Bbb C^r\setminus {\rm diagonals}$ with values in $(V_1\otimes...\otimes V_r)[\mu]$ has the form 
\begin{equation}\label{kzconn}
\nabla(\hbar):=d-\hbar\sum_{i=1}^r\left(\sum_{j\ne i}\frac{\Omega^{ij}}{x_i-x_j}\right) dx_i, 
\end{equation} 
where 
$\Omega^{ij}$ is $\Omega$ acting in the $i$-th and $j$-th factor.  
It is easy to see that this connection has regular singularities.

By the Drinfeld-Kohno theorem (\cite{Dr},\cite{EK})\footnote{At the end of \cite{Dr}, Drinfeld claims this theorem for finite dimensional representations of $\g$, but the proof applies to the more general case of category $\mathcal O$ using the result of \cite{EK} that the Etingof-Kazhdan functorial quantization of a simple Lie algebra with its standard quasitriangular structure is isomorphic to the usual Drinfeld-Jimbo quantization (for formal $\hbar$).}, the monodromy of $\nabla(\hbar)$ for all $\hbar$ except possibly a countable set is given by $R$-matrices of the quantum group $U_q(\g)$, where $q=e^{\pi i\hbar}$. So, given that the weights of $V_j$ are rational, the monodromy depends 
on $e^{\pi i\hbar/T}$, where $T$ is the common denominator of the weights of $V_j$. 

Variants of this construction include restricting the KZ connection to singular vectors 
(i.e., $\mathfrak n_+$-invariants in $(V_1\otimes...\otimes V_r)[\mu]$) or, for $\mu=0$, $\g$-invariants. 

This example can be generalized in several directions described in the following subsections, which gives many more examples of pencils of regular flat connections that turn out to be quasi-motivic.

\subsubsection{KZ connections for Kac-Moody algebras} \label{kz2}
In \ref{kz1}, the Lie algebra $\g$ may be replaced with 
a symmetrizable Kac-Moody algebra $\g(A)$ for a generalized Cartan matrix $A$, or, more generally, with the Lie algebra $\widetilde \g(A)$ without Serre relations attached to a rational matrix $A$ (see \cite{V}). In this case, the relevant version of the Drinfeld-Kohno theorem is proved in \cite{V} (for Verma modules and integrable modules) and \cite{EK} (in general). 

\subsubsection{Multiparameter KZ connections} \label{kz3}
We may consider a multiparameter version of the pencil \ref{kz1}. For simplicity 
consider the case when $V_i=M(\lambda_i)$ are Verma modules with highest weights $\lambda_i\in \mathfrak h^*$, and $\mu=\sum_{i=1}^r \lambda_i-\beta$, where $\beta\in Q_+$. In this case, we have a natural identification
$$
(V_1\otimes...\otimes V_r)[\mu]\cong \bigoplus_{\beta_1,...,\beta_r: \sum_{i=1}^r\beta_i=\beta}\bigotimes_{i=1}^r U(\mathfrak n_-)[\beta_i].
$$
Realizing $M(\lambda)\otimes M(\mu)$ as $U(\n_-)\otimes U(\n_-)$, 
we can write the operator $\Omega$ on $M(\lambda)\otimes M(\mu)$ as 
$$
\Omega=(\lambda,\mu)+Q+\sum_k \lambda(h_k)P_k^{12}+\sum_k \mu(h_k)P_k^{21},
$$
where $Q,P_k$ do not depend on $\lambda,\mu$ and have rational coefficients. It follows that the KZ connection can be written as 
$$
\nabla=d-\hbar\sum_{i=1}^r\left(\sum_{j\ne i}\frac{(\lambda_i,\lambda_j)+Q^{ij}+\sum_k\lambda_i(h_k)(P_{k}^{ij}+P_k^{ji})}{x_i-x_j}\right)dx_i.
$$
Define the modified KZ connection by conjugating $\nabla$ by the multivalued analytic function 
$\psi_0=\prod_{i<j}(x_i-x_j)^{\hbar(\lambda_i,\lambda_j)}$: 
$$
\nabla_*=d-\hbar\sum_{i=1}^r\left(\sum_{j\ne i}\frac{Q^{ij}+\sum_k\lambda_i(h_k)(P_{k}^{ij}+P_k^{ji})}{x_i-x_j}\right)dx_i.
$$
Now we can define $\bold s$ to be the vector with coordinates 
$s_0:=\hbar$ and $s_{jk}:=\hbar \lambda_j(h_k)$, and we get 
a pencil of flat connections $\nabla_*(\bold s)$. Moreover, the Drinfeld-Kohno theorem 
(\cite{Dr},\cite{EK}) implies that the monodromy of this pencil depends on 
$q=e^{\pi i\hbar}$ and $e^{\pi is_{jk}}=q^{\lambda_j(h_k)}$. 

\begin{example}\label{sl2ex} Let $\g=\mathfrak{sl}_2$. In this case we have $M(\lambda)=\Bbb C[f]v_{\lambda}\cong \Bbb C[z]$. Thus, setting $\partial=\partial_z$, we have $\Omega=\frac{1}{2}h\otimes h+e\otimes f+f\otimes e$, so
by direct computation we obtain
$$
\Omega|_{M(\lambda)\otimes M(\mu)}=\tfrac{1}{2}\lambda\mu+\lambda(\partial \otimes z-1\otimes z\partial)+\mu(z\otimes \partial-z\partial \otimes 1)-z\partial^2\otimes 1-1\otimes z\partial^2.
$$
Thus 
$$
P=\partial \otimes z-1\otimes z\partial,\ Q=-z\partial^2\otimes 1-1\otimes z\partial^2.
$$
\end{example} 

\begin{example} Consider the multiparameter KZ pencil for $\g=\mathfrak{sl}_2$
in the vector space $(M(\lambda_1)\otimes...\otimes M(\lambda_r))[\sum_{j=1}^r \lambda_j-2]\cong \Bbb C^r$. 
Let $s_j:=\hbar\lambda_j$. Then from Example \ref{sl2ex} we get that the KZ equations for a flat section $I(\bold s,\bold x)=(I_1(\bold s,\bold x),...,I_r(\bold s,\bold x))$ of $\nabla(\bold s)$ (conjugated by $\psi_0$) have the form 
\begin{equation}\label{kze}
\frac{\partial I_j}{\partial x_i} =
s_i  \frac{I_i- I_j}{x_i-x_j},\ j\ne i; \quad
\frac{\partial I_i}{\partial x_i} 
= -\sum_{j\ne i}s_j \frac{I_i- I_j}{x_i-x_j}\,.
\end{equation}
(In this special case the KZ connection does not explicitly depend on $\hbar$ but only depends on $s_j$, because $Q=0$). 
We have a $\nabla(\bold s)$-invariant decomposition $\Bbb C^r=\Bbb C\bold 1\oplus V_{\bold s}$, where $\bold 1=(1,...,1)$ and $V_{\bold s}$ is the orthogonal complement to $\bold s$.  
Thus the shift operators $A_k$ should decompose accordingly as
$$
A_k(\bold s,\bold x)=A_k^0(\bold s,\bold x)\oplus \xi_k(\bold s)(\bold 1\otimes \bold s^T),
$$
where $\xi_k$ are scalar rational functions and 
$$
A_k^0(\bold s,\bold x) \bold 1=0,\ A_k^0(\bold s,\bold x)(V_{\bold s})\subset V_{\bold s+\bold e_k}.
$$
Then $A_k^0$ is uniquely determined up to scaling by a rational function of $\bold s$. 

To compute $A_k^0$, note that by definition $A_k^0(\bold s,\bold x)$
expresses flat sections of $\nabla(\bold s+\bold e_k)$ via flat sections 
of $\nabla(\bold s)$. Now, flat sections of $\nabla(\bold s)$ are solutions of \eqref{kze}, so they have the form 
$$
I_i({\bold x},{\bold s})  =\int  \Phi(t,\bold s,\bold x)  \frac{dt}{t-x_i},
$$
where 
$$
\Phi(t,{\bold s},{\bold x}):=\prod_{j=1}^r(t-x_j)^{-s_j}
$$
is the master function, and integration is over a suitable twisted cycle (see e.g. \cite{EFK}, 4.3). Thus for $i\ne k$, we get by a direct computation
$$
I_i(\bold s+\bold e_k,\bold x) =
 \int\Phi(t,\bold s,\bold x)\frac1{(t-x_k)(t-x_i)}dt =
\frac{I_i(\bold s,\bold x)-I_k(\bold s,\bold x)}{x_i-x_k}.
$$
and
$$
I_k(\bold s+\bold e_k,\bold x) 
=
\int\Phi(t,\bold s,\bold x)\frac1{(t-x_k)^2}dt =
  -\sum_{i\ne k} \frac{s_i}{s_k+1} \frac{I_i(\bold s,\bold x)  - I_k(\bold s,\bold x)}{x_i-x_k}.
$$
This yields the following formula for $A_k^0$: 
$$
A_k^0(\bold s,\bold x)=\sum_{i\ne k}\frac{(s_k+1)(E_{ii}-E_{ik})+s_i(E_{kk}-E_{ki})}{x_i-x_k},
$$
where $E_{ij}$ are the elementary matrices. 
\end{example} 

\begin{remark}\label{contmat} The discussion of this subsection applies mutatis mutandis 
to the setting of \ref{kz2}. 
Moreover, one may consider the case of a generic complex matrix $A$ and 
Lie algebra $\widetilde \g(A)$ without Serre relations. In this case the entries $a_{ij}$ of $A$ multiplied by $\hbar$ should become parameters of the pencil (which can be done since the KZ connection is linear in $a_{ij}$). 
\end{remark} 

\begin{remark} The integral formulas for solutions of KZ equations given in \cite{SV},\cite{V} imply that the pencils in \ref{kz1}, \ref{kz2},\ref{kz3} are actually motivic (not just quasi-motivic). 
\end{remark}

\subsubsection{KZ  connections in the Deligne category} 
For $t\in \Bbb C$ let ${\rm Rep}(GL_t)$ be the (abelian) Deligne category (\cite{EGNO}, Section 9.12),
and let $V=[1,0]$ be the generating object (the ``vector representation").
Consider the space ${\rm Hom}(\bold 1,V^{\otimes m}\otimes V^{*\otimes m})={\rm End}(V^{\otimes m})$. By the Schur-Weyl duality, this space is naturally identified with $\Bbb CS_m$. 
The Casimir operator $\Omega$ acts on $V\otimes V$ and $V^*\otimes V^*$ by the flip
and on $V\otimes V^*$ by $-{\rm coev}\circ {\rm ev}$, where 
${\rm ev}: V\otimes V^*\to \bold 1$, ${\rm coev}: \bold 1\to V\otimes V^*$ 
are the evaluation and coevaluation maps. Thus the operators 
$\Omega^{ij}$, $1\le i<j\le 2m$ on ${\rm Hom}(\bold 1,V^{\otimes m}\otimes V^{*\otimes m})=\Bbb C S_m$ 
act as follows: for $\sigma\in S_m$, 
$$
\Omega^{ij}\sigma=\begin{cases}(i,j)\circ \sigma\text{ if } i,j\le m;\\ \sigma\circ (i-m,j-m)\text{ if } i,j>m;\\ -t\sigma \text{ if } \sigma(j-m)=i,\ i\le m<j;\\ -(i,\sigma(j-m))\circ \sigma=-\sigma\circ (\sigma^{-1}(i),j-m)\text{ if } \sigma(j-m)\ne i,\ i\le m<j,\end{cases}
$$
where $(i,j)$ is the transposition of $i$ and $j$. In other words, 
$\Omega^{ij}=\Omega^{ij}_0+t\Omega^{ij}_1$, where 
$\Omega^{ij}_0,\Omega^{ij}_1$ are independent on $t$. 

We may now set $s_0=\hbar$, $s_1=\hbar t$ and 
consider the pencil of flat connections $\nabla_{KZ,m}(\bold s)$ with fiber $\Bbb C S_m$ given by \eqref{kzconn}. Note that if $t=n\ge m$ is a positive integer then this connection 
coincides with the usual KZ  connection for 
$\mathfrak{gl}_n$ on $(V^{\otimes m}\otimes V^{*\otimes m})^{\mathfrak{gl}_n}$.
Thus, repeating the arguments of \cite{Dr} in the Deligne category (which can be done since they are based on diagrammatic tensor calculus), we can interpolate the Drinfeld-Kohno theorem to non-integer $t$. This implies that the monodromy of $\nabla_{KZ,m}$ is given by the braiding operators 
of the tensor category ${\rm Rep}_q(GL_t)$, also called the oriented skein category (see e.g. \cite{B}). These operators are rational functions in $q:=e^{\pi is_0}$ and $a:=q^t=e^{\pi is_1}$. 

\subsubsection{KZ connections for Lie superalgebras} 

The results of \ref{kz1} can be generalized to the case when $\g$ 
is a Lie superalgebra of type A-G, and $V_1,...,V_r$ are finite dimensional $\g$-modules. 
In this case the relevant version of the Drinfeld-Kohno theorem is proved in \cite{G}: 
again it turns out that the monodromy representation depends on $q:=e^{\pi i \hbar}$.  

This can be generalized in several directions: 

1. $V_1,...,V_r$ can be taken from category $\mathcal O$ for some Borel subalgebra 
of $\g$ (with rational weights), as in \ref{kz1}. Note that for Lie superalgebras there are several types of Borel subalgebras up to conjugation. 

2. The setup can be generalized to Kac-Moody Lie superalgebras, as in \ref{kz2} (for basic results on their structure and representation theory we refer the reader to \cite{S}).  

3. $V_1,...,V_r$ can be taken to be Verma modules with complex highest weights for some Borel subalgebra, and their highest weights multiplied by $\hbar$ can be made parameters of a multiparameter pencil, as in \ref{kz3}. 

\begin{remark} To be more precise, the Lie superalgebras of type A-G contain a unique continuous family 
$D(2,1,\alpha)$, $\alpha\in \Bbb C$. So the above discussion applies literally when $\alpha\in \Bbb Q$. For more general $\alpha\in \Bbb C$, we should make it a parameter of the pencil, similarly to Remark \ref{contmat} (as the entries $a_{ij}$ of the Cartan matrix of $D(2,1,\alpha)$ depend linearly on $\alpha$). 
\end{remark} 

\subsubsection{Trigonomeric KZ connections} 

Let $\g$ be a simple Lie algebra as in \ref{kz1} and $\bold r\in \g\otimes \g$ a quasitriangular structure (classical $r$-matrix) such that $\bold r+\bold r^{21}=\Omega$. Such structures were classified by Belavin and Drinfeld, \cite{BD} and are labeled (up to abelian twists) by Belavin-Drinfeld triples. Given $\bold r$, we may define the affine r-matrix (with spectral parameter)
$$
\widehat{\bold r}(z):=\frac{\bold r^{21}z+\bold r}{z-1}.
$$
We may assume that the centralizer of $\bold r$ in $\g$ is a Lie subalgebra $\mathfrak h_{\bold r}\subset \mathfrak h$ (this can always be achieved by conjugation). Let $s\in \mathfrak h_{\bold r}$, let $V_1,...,V_r$ be finite dimensional $\g$-modules, and consider the connection on the trivial bundle with fiber $(V_1\otimes...\otimes V_r)[\mu]$, $\mu\in\mathfrak h_{\bold r}^*$ on $\Bbb C^r$ with first order poles at $x_i=0$ and $x_i=x_j$, given by 
\begin{equation}\label{kzconn1}
\nabla(\hbar,s):=d-\sum_{i=1}^r\left(s^{(i)}+\hbar\sum_{j\ne i}\widehat{\bold r}(x_i/x_j)\right)\frac{dx_i}{x_i}. 
\end{equation} 
This connection is flat and called {\bf the trigonometric KZ connection}; 
it is a pencil with parameters $\bold s=(\hbar,s)\in \Bbb C\oplus \mathfrak h_{\bold r}$ (i.e., $\hbar=s_0, s=(s_1,...,s_n)$). 
The monodromy of the trigonometric KZ connection is computed in \cite{EG} in terms of the Etingof-Kazhdan quantization of the quasitriangular Lie bialgebra $(\g,\bold r)$, using its relationship to the usual KZ connection. This quantization is described explicitly in \cite{ESS}, and this description implies that the monodromy of the trigonometric KZ connection 
depends on $q=e^{\pi i\hbar}$ and $q^s$. 

\begin{remark} 1. In the case when $\bold r$ is the standard classical $r$-matrix of $\g$, the trigonometric KZ connection can be easily obtained from the rational one, as explained for example in \cite{EFK}, Subsection 3.8, or \cite{MV}. In particular, the integral representations of solutions of the KZ equations from \cite{SV} give integral representations of solutions of the
trigonometric KZ equations. Hence the trigonometric KZ connection for the standard $r$-matrix is 
motivic.

2. Moreover, in this case the shift operators $A_j(\bold s)$ in the parameters $s_j$, $j\ge 1$ are the dynamical difference operators introduced in \cite{TV} and studied further in \cite{EV} (the lattice part of the dynamical Weyl group for the affine Lie algebra $\widehat \g$). However, the shift operator $A_0(\bold s)$ in $s_0=\hbar$ is more complicated, and its explicit expression is not known in general. 
\end{remark}

As before, this can be generalized to category $\mathcal O$ modules and to Lie superalgebras. 

Another generalization arises if the trigonometric $r$-matrix $\widehat {\bold r}(z)$ 
corresponding to ${\bold r}$ is replaced by an arbitrary 
trigonometric r-matrix $\widehat{\bold r}(z)$ from \cite{BD1} (i.e., corresponding to a Belavin-Drinfeld triple for a possibly twisted affine Lie algebra). This defines a pencil of flat connections \eqref{kzconn1} (where $s\in \mathfrak h$ preserves $\widehat{\bold r}$)  whose monodromy can be computed using the trace interpretation of flat sections given in \cite{ES}, and shown to depend on $q=e^{\pi i\hbar}$ and $q^s$.

\subsubsection{Elliptic KZ connections} Let $\tau\in \Bbb C_+$ and 
$E=\Bbb C^\times/e^{2\pi i\tau\Bbb Z}$
be an elliptic curve. Let $\g=\mathfrak{sl}_n$ and $\bold r(z,\tau)$ 
be the Belavin elliptic $r$-matrix associated to some 
integer $1\le k\le n-1$ coprime to $n$ (\cite{BD1}); it has simple poles at points of order $n$ on $E$. Let $V_1,...,V_r$ be finite dimensional representations of 
$\mathfrak{sl}_n$. The {\bf elliptic KZ connection}
is the connection on the trivial bundle on $E^r$ with fiber $V_1\otimes...\otimes V_r$ 
and first order poles at $x_i=\zeta x_j$, where $\zeta\in E$, $\zeta^n=1$, given by the formula 
$$
\nabla(\hbar)=d-\hbar\sum_{i=1}^r\sum_{j\ne i}{\bold r}(x_i/x_j,\tau)\frac{dx_i}{x_i}. 
$$
This is a pencil of flat connections on its regular locus $X\subset E^r$. 
The monodromy of this pencil was computed in \cite{E} and shown 
to depend on $q=e^{\pi i\hbar}$. 

\subsection{Dunkl connections} 

\subsubsection{Dunkl connections} \label{ratdun}

Let $W$ be an irreducible finite Coxeter group with reflection representation $\mathfrak h$, $S\subset W$ the set of reflections, and for any $w\in S$ let $\alpha_w\in \mathfrak h^*\setminus 0$ be such that $w\alpha_w=-\alpha_w$. Let $\mathfrak h_{\rm reg}\subset \mathfrak h$ be
the set of $\bold x\in \mathfrak h$ such that $\alpha_w(\bold x)\ne 0$ for all $w\in S$. Let $V$ be a finite dimensional representation of $W$. Consider the pencil of connections 
$$
\nabla(c)=d-c \sum_{w\in S}\frac{d \alpha_w}{\alpha_w}(w-1)
$$
on the trivial bundle with fiber $V$ over $\mathfrak h_{\rm reg}$. 
It was proved by Dunkl that $\nabla(c)$ is flat, and it is called the {\bf Dunkl connection}. 
It is not hard to show that the Dunkl connection has regular singularities and its monodromy is given by the representation
$V_q$ of the {\bf Hecke algebra} $H_q(W)$ attached to $V$, where $q=e^{2\pi ic}$
(see \cite{GGOR}). This representation is determined by $q$, or in rare cases by $q^{\frac{1}{2}}$, due to presence of Opdam's KZ twists (see e.g. \cite{BC}, 7.2). 

This can be generalized in several directions. First of all, if $W$ is of type $I_{2m}$ (even dihedral type), $B_m$, $m\ge 3$, or $F_4$, then $S$ falls into two conjugacy classes $S_1,S_2$, so we can consider the 2-parameter pencil of Dunkl connections
$$
\nabla(c_1,c_2)=d-c_1 \sum_{w\in S_1}\frac{d \alpha_w}{\alpha_w}(w-1)-c_2 \sum_{w\in S_2}\frac{d \alpha_w}{\alpha_w}(w-1).
$$
The above discussion extends verbatim to this case. 

Secondly, we can generalize this story to the case when $W$ is a complex reflection group. 
In this case we should fix a conjugation-invariant function $c: S\to \Bbb C$, and 
the Dunkl connection (introduced in this generality in \cite{BMR,DO}) has the form 
$$
\nabla_c=d-\sum_{w\in S}c(w)\frac{d \alpha_w}{\alpha_w}(w-1).
$$
Let $\overline S$ be the set reflection hyperplanes, 
and pick representatives $E_1,...,E_m$ of $W$-orbits in $\overline S$.
For each $1\le k\le m$, the stabilizer of a generic point in $E_k$ is 
$\Bbb Z/\ell_k$, generated by an element $w_k$ which has nontrivial eigenvalue 
$e^{\frac{2\pi i}{\ell_k}}$ on $\mathfrak h$. Let $c_{jk}:=c(w_k^j)$ 
where $1\le j\le \ell_k-1$. Then $c=(c_{jk})$. Let us make the following linear 
transformation of these coordinates: 
$$
s_{lk}:=\sum_{j=1}^{\ell_k-1}e^{\frac{2\pi ijl}{\ell_k}}c_{jk},\ 0\le l\le \ell_k-1.
$$
Thus for each $k$ we have $\sum_{l=0}^{\ell_k-1}s_{lk}=0$. 
Let $\bold s=(s_{lk})$ (subject to these conditions) and $\nabla(\bold s):=\nabla_c$. 
It is shown in \cite{GGOR} that the monodromy of $\nabla(\bold s)$ factors through the 
{\bf cyclotomic Hecke algebra} $H_{\bold q}(W)$ where $\bold q=e^{2\pi i\bold s}$, and it is now known that, as originally conjectured in \cite{BMR}, this algebra is a flat deformation of the group algebra $\Bbb CW$, i.e. has dimension $|W|$ (see \cite{E1} and references therein). 

\subsubsection{Thrigonometric Dunkl (Dunkl-Cherednik) connections}\label{trigdun}

Now let $W$ be the Weyl group of an irreducible root system $R\subset \mathfrak h^*$ of rank $r$ with a polarization $R=R_+\cup R_-$. For $\alpha\in R_+$ let $s_\alpha\in W$ be the reflection corresponding to $\alpha$. Let $Q\subset \mathfrak h^*$ be the root lattice of $R$ spanned by $R$ over $\Bbb Z$ and $P^\vee\subset \mathfrak h$ be the dual lattice of coweights. Let $H=\mathfrak h/2\pi i P^\vee$ be the maximal torus of the corresponding adjoint simple complex Lie group $G$.  For $\alpha\in Q$ denote by $e^\alpha$ the character of $H$ corresponding to $\alpha$. Let $H_{\rm reg}$ be the set of $g\in H$ such that $e^{\alpha}(g)\ne 1$ for any $\alpha\in R_+$.

Following Drinfeld and Lusztig, define the {\bf degenerate affine Hecke algebra} 
$\mathcal H(W)$ to be the algebra generated by $W$ and ${\rm Sym}\mathfrak h$ 
with defining commutation relations
$$
s_\alpha h-s_{\alpha}(h)s_\alpha=\alpha(h),\ h\in \mathfrak h
$$
for simple roots $\alpha$. 
It is well known that the multiplication map $\Bbb CW\otimes  {\rm Sym}\mathfrak h\to \mathcal H(W)$ in either order is an isomorphism of vector spaces (the PBW theorem for $\mathcal H(W)$, see e.g. \cite{L}). 
Let $V$ be a finite dimensional representation of $\mathcal H(W)$.  
Consider the pencil of connections on the trivial bundle over $H_{\rm reg}$ with fiber $V$ given by 
$$
\nabla(c)=d-c\left(\sum_{i=1}^r \omega_i^\vee d\alpha_i+\sum_{\alpha\in R_+}\frac{e^{\alpha}d\alpha}{1-e^{\alpha}}(s_\alpha-1)\right), 
$$
where $\alpha_i$ are the simple roots and $\omega_i^\vee$ the fundamental coweights 
(here $s_\alpha$ and $\omega_i^\vee$ act on $V$ as elements of $\mathcal H(W)$). In other words, 
using $x_i:=e^{\alpha_i}$ as coordinates on $H$, we obtain the following formula for the covariant derivatives:
$$
\nabla_i(c)=\partial_i-\frac{c}{x_i}\left(\omega_i^\vee+\sum_{\alpha\in R_+}\alpha(\omega_i^\vee)\frac{\prod_{j=1}^r x_j^{\alpha(\omega_j^\vee)}}{1-\prod_{j=1}^r x_j^{\alpha(\omega_j^\vee)}}(s_\alpha-1)\right).
$$
This connection is flat, has regular singularities, 
and is called the {\bf Dunkl-Cherednik connection} (also known as trigonometric Dunkl connection, or the affine KZ connection for $R$); it was introduced in \cite{Ch}. Cherednik showed that 
the monodromy of this connection is given by the representation $V_q$ of the {\bf affine Hecke algebra} $\mathcal H_q(W)$ corresponding to $V$ via Lusztig's map, \cite{L}, where $q=e^{2\pi ic}$. 

This can be upgraded to a multiparameter version. To this end, define the family 
of {\bf induced representations} $V_\lambda$ of $\mathcal H(W)$ for a weight $\lambda$. Namely, given $\lambda\in \mathfrak h^*$, let $V_\lambda:=\mathcal H(W)\otimes_{{\rm Sym}\mathfrak h}\Bbb C_\lambda$, where $\mathfrak h$ acts in $\Bbb C_\lambda$ by $h\mapsto \lambda(h)$. The PBW theorem for $\mathcal H(W)$ yields an isomorphism 
$V_\lambda\cong \Bbb CW$ as a $\Bbb CW$-module. Moreover, the commutation relations 
of $\mathcal H(W)$ imply that for any $h\in \mathfrak h$, the action of $h$ in $V_\lambda$ 
is given by 
$$
h\circ w=\lambda(w^{-1}h)w+\xi_w(h) 
$$
where $\xi_w(h)$ comprises the lower length terms independent of $\lambda$. 

Thus the pencil $\nabla(c)$ on $V_\lambda$ may be regarded as a multiprameter pencil 
$\nabla(\bold s)$ with parameter $\bold s:=(c,c\lambda)$. 
Cherednik's monodromy theorem implies that the monodromy 
of $\nabla(\bold s)$ depends of $e^{2\pi i\bold s}$. 

As in the rational case, this can be generalized to 
unequal parameters for different $W$-orbits of roots. 
Let us present this generalization in the multiparameter form. 
Namely, let $R^l$ be the $W$-orbits of $R$, and $c_l$ 
the corresponding parameters. In the case of several parameters the degenerate affine Hecke algebra $\mathcal H_{\bold c}(W)$ depends on $\bold c$; namely it is the algebra generated by $W$ and ${\rm Sym}\mathfrak h$ with defining commutation relations
$$
s_\alpha h-s_{\alpha}(h)s_\alpha=c_l\alpha(h),\ h\in \mathfrak h
$$
for simple roots $\alpha\in R_l$. Let $V_{\bold c,\lambda}$ be the 
induced representation $\mathcal H_{\bold c}\otimes_{{\rm Sym}\mathfrak h}\Bbb C_\lambda$.
Note that for any $t\ne 0$ we have a natural isomorphism $\mathcal H_{\bold c}(W)\cong \mathcal H_{t\bold c}(W)$ which preserves $w\in W$ and maps $h\in \mathfrak h$ to $th$, and that 
under this isomorphism $V_{\bold c,\lambda}$ maps to $V_{t\bold c,t\lambda}$. 
Thus in the above single-parameter case for $c\ne 0$, the representation $V_{c,\lambda}$ 
of $\mathcal H_c(W)$ corresponds to the representation $V_{\lambda/c}$ of $\mathcal H(W)$. 
As before, $V_{\bold c,\lambda}$ is naturally identified with $\Bbb CW$ as a 
$\Bbb CW$-module with 
\begin{equation}\label{lin}
h\circ w=\lambda(w^{-1}h)w+\sum_l c_l\xi_{l,w}(h), 
\end{equation} 
where $\xi_{l,w}:\mathfrak h\to \Bbb CW$ is a linear map (landing in the span of elements 
of length $<\ell(w)$). 
Then the Dunkl-Cherednik connection has the form
$$
\nabla(\bold c)=d-\sum_{i=1}^r \omega_i^\vee d\alpha_i-\sum_l c_l\sum_{\alpha\in R^l_+}\frac{e^{\alpha}d\alpha}{1-e^{\alpha}}(s_\alpha-1). 
$$
As before, $\nabla(\bold c)$ is flat and has regular singularities, and by Cherednik's theorem 
its monodromy on $V_{\bold c,\lambda}$ depends on $e^{2\pi i \bold c}$ and $e^{2\pi i\lambda}$. 
Also in view of \eqref{lin}, $\nabla(\bold c)|_{V_{\bold c,\lambda}}$ is a pencil $\nabla(\bold s)$ with parameter $\bold s=(\bold c,\lambda)$. 

\begin{remark} In type $A$ ($W=S_n$) the (trigonometric) Dunkl connections are 
special cases of the (trigonometric) KZ connections for $\g=\mathfrak{sl}_n$.
Namely, we should take $V_1,...,V_n$ to be the vector representation $\Bbb C^n$ 
and consider the connection on $(V_1\otimes...\otimes V_n)[0]$. 
\end{remark} 

 \subsection{Casimir connections} 
 
 \subsubsection{Rational Casimir connections} \label{casconn} Consider the setting of \ref{kz1}. Let $\mathfrak h_{\rm reg}$ be the regular part of the Cartan subalgebra $\mathfrak h\subset \g$ as in \ref{ratdun}.
 Let $V$ be a finite dimensional $\g$-module. The {\bf Casimir connection} (also called the {\bf dynamical connection}) is the connection on the trivial bundle on $\mathfrak h_{\rm reg}$ with fiber $V$ given by 
$$
\nabla(\hbar)=d-\hbar\sum_{\alpha\in R_+}\frac{e_{-\alpha} e_\alpha+e_{-\alpha}e_\alpha}{2}\frac{d\alpha}{\alpha},
$$
where $e_\alpha$ are the root elements of $\g$ defined in \ref{kz1}. 
This connection is flat and has regular singularities. It was first defined by De Concini around 1995 (unpublished), and then  
was studied in the papers \cite{FMTV}, \cite{TL}, \cite{MTL} and many subsequent ones. In particular, Toledano-Laredo proved (\cite{TL1},\cite{TL2}) that the monodromy of 
$\nabla(\hbar)$ is given by the quantum Weyl group action on the representation 
$V_q$ of the quantum group $U_q(\g)$ corresponding to $V$, where $q=e^{\pi i\hbar}$ (as had been conjectured independently by Toledano-Laredo and C. De Concini and proved for $\g=\mathfrak{sl}_n$ in \cite{TL}). Since $\nabla$ is $\mathfrak h$-invariant, the same is clearly true on every weight subspace $V[\mu]$.  
 
This can be generalized in various ways. First of all, one can take $V$ 
to be a module from category $\mathcal O$ with rational weights and restrict 
$\nabla(\hbar)$ to a weight subspace $V[\mu]$. The monodromy of $\nabla$ is 
then still given by the (pure) quantum Weyl group (\cite{ATL}).

Also we can consider the multiparameter version. To this end take 
$V=M(\lambda)$, the Verma module over $\g$ with highest weight $\lambda$, 
and take $\mu=\lambda-\beta$ for fixed $\beta\in Q_+$. Then 
$V[\mu]\cong U(\mathfrak n_-)[-\beta]$, and in this realization 
the operators $e_{-\alpha} e_\alpha$ depend linearly (inhomogeneously) on $\lambda$. 
Thus $\nabla(\hbar)|_{M(\lambda)[\lambda-\beta]}$ can be viewed as a multiparameter pencil 
with poarameter $\bold s=(\hbar,\hbar \lambda)$. By \cite{ATL}, the monodromy of $\nabla(\bold s)$ 
is given by the pure quantum Weyl group, so depends on $e^{2\pi i\bold s}$.

Finally, this whole story extends to the case when $\g$ is a symmetrizable Kac-Moody algebra, since this is the generality of the monodromy result of \cite{ATL}. 

\subsubsection{Trigonometric Casimir connections} \label{trigcas}

The trigonometric version of the Casimir connection was introduced in \cite{TL3}. 
To define it, keep the setting of the previous subsection and the notation of \ref{trigdun}. 
Let $Y(\g)$ be Drinfeld's {\bf Yangian} 
of $\g$  in which the Planck constant $\hbar$ is set to $1$, and let $J: \g\to Y(\g)$, be the $\g$-invariant map defining the Drinfeld generators of $Y(\g)$ (\cite{TL3}, Section 3). Let $V$ be a finite dimensional irreducible representation of $Y(\g)$ on which $J(\omega_i^\vee)$ acts with rational eigenvalues (i.e., its Drinfeld polynomial has rational roots). 
The {\bf trigonometric Casimir connection} is the connection on the trivial bundle on $H_{\rm reg}$ with fiber $V$ given by 
$$
\nabla(\hbar) = d-\hbar\left(\sum_{\alpha\in R_+}\frac{1}{4}\frac{e^\alpha+1}{e^\alpha-1}(e_\alpha e_{-\alpha}+e_{-\alpha} e_\alpha)d\alpha -\sum_{i=1}^{{\rm rank}(\g)}J(\omega_i^\vee)d\alpha_i\right).
$$
It is shown in \cite{TL3} that this connection is flat, and it clearly has regular singularities. Note that $\nabla$ falls into a direct sum of connections on $\mathfrak h$-weight spaces of $V$. Also $\nabla$ is a pencil, and it arises as the equivariant quantum connection for quiver varieties of finite type (\cite{MO}), hence is periodic and quasi-motivic.  

Also for $\g=\mathfrak{sl}_2$ it is shown in \cite{GTL} that the monodromy of $\nabla(\hbar)$ is given by the action of the quantum Weyl group on the corresponding representation $V_q$ of the quantum affine algebra $U_q(\widehat \g)$, and the proof for general $\g$ 
is forthcoming in a joint work of Gautam and Toledano-Laredo. 
This also will imply the quasi-motivic and periodic properties of $\nabla(\hbar)$.  

There is also a multiparameter version of this pencil. To describe it, denote by $V(z)$ 
the shift of the representation $V$ by $z\in \Bbb C$, which is $V$ as a $\g$-module with 
$$
J(a)|_{V(z)}=(J(a)+za)|_V.
$$ 
Let $V_1,...,V_n$ be irreducible finite dimensional $Y(\g)$-modules whose Drinfeld polynomials have rational roots, and let $V:=V_1(z_1)\otimes...\otimes V_n(z_n)$ (tensor product using the Hopf structure on $Y(\g)$). 
Recall that 
$$
\Delta(J(h))=J(h)\otimes 1+1\otimes J(h)+\frac{1}{2}[h\otimes 1,\Omega].
$$
Thus $V=V_1\otimes...\otimes V_n$ as a $\g$-module, and 
$$
J(h)|_V=\sum_{j=1}^n z_j(J(h)|_{V_j})^{(j)}+J_0(h)|_{V_1\otimes...\otimes V_n},
$$
where $J_0(h)\in U(\g)^{\otimes n}$ does not depend on $z_j$. Thus the pencil $\nabla(\hbar)|_V$ 
can be viewed as a multiparameter pencil $\nabla(\bold s)$ with parameter $\bold s=(\hbar,\hbar \bold z)$, where $\bold z:=(z_1,...,z_n)$. This pencil arises as the equivariant quantum connection for quiver varieties of finite type (\cite{MO}), so it is quasi-motivic and periodic. 
This will also follow from the forthcoming results of Gautam and Toledano-Laredo on the monodromy of $\nabla(\bold s)$. 

\begin{remark} 1. If $\g=\mathfrak{sl}_n$ then the (trigonometric) Casimir connection 
is equivalent to the (trigonometric) KZ connection for $\mathfrak{gl}_m$, via the Howe duality 
between $\mathfrak{gl}_n$ and $\mathfrak{gl}_m$, see \cite{GTL}, \cite{TL}. 
As explained in these papers, this duality can be used to compute the monodromy 
of the trigonometric Casimir connection in type $A$. 

2. Let $V=\g$ be the adjoint representation of $\g$. Then the Casimir connection on $V[0]$ 
coincides with the Dunkl connection for the Weyl group $W$ of $\g$ in its reflection representation
$\mathfrak h$. Similarly, in the trigonometric case consider the adjoint representation 
$V$ of the Yangian $Y(\g)$ (evaluated at $0$), which is $\g$ for type $A$ and $\g\oplus \Bbb C$ otherwise as a $\g$-module. Then the trigonometric Casimir connection in $V[0]$ coincides with the trigonometric Dunkl connection 
on the reflection representation of the degenerate affine Hecke algebra $\mathcal H(W)$, which is 
$\mathfrak h$ in type $A$ and $\mathfrak h\oplus \Bbb C$ otherwise. 
\end{remark} 

\subsection{Irreducibility} In many of the cases described above, the 
connection $\nabla(\bold s)$ is irreducible for generic $\bold s$, so that 
Theorem \ref{badval}(ii) applies. For example, it follows from Theorem 1.2 in \cite{MTL} and a result of the first author (included in \cite{MTL} as Theorem 1.8) that the Casimir connection $\nabla(\hbar,\hbar\mu)$ on the weight subspace $M(\mu)[\mu-\beta]$ in the Verma module $M(\mu)$ is irreducible for Zariski generic $(\hbar,\mu)$. Another example where the connection is generically irreducible is the Dunkl connection on an irreducible 
$W$-module $V$ (this follows easily by the same argument as in the proof of \cite{MTL}, Theorem 1.2).

\subsection{Unitarity and generic semisimplicity} 

In many of the above examples, the monodromy representation of the connection is unitary, 
hence semisimple for small real values of $\bold s$. Thus such pencils are generically semisimple. 

For example, consider the KZ connection \eqref{kzconn} in ${\rm Hom}_\g(V_0,V_1\otimes...\otimes V_r)$ 
where $V_i$ are finite dimensional representations of $\g$. 
As noted above, the Drinfeld-Kohno theorem identifies 
the monodromy of the KZ connection with the
representation of the pure braid group $PB_r$
by $R$-matrices of the quantum 
group $U_q(\g)$, where 
$q=e^{\pi i\hbar}$. If $\hbar$ is real (i.e., $|q|=1$), this quantum group 
has a real structure, and its category of representations is a Hermitian 
tensor category (when $q$ is not a root of $1$). This means that 
multiplicity spaces of tensor products of representations 
of $U_q(\g)$ carry a nondegenerate Hermitian inner product. Moreover, if $\hbar=0$, this is the usual 
positive inner product on multiplicity spaces for the corresponding compact Lie group. 
Thus by deformation argument, when the representations $V_i$ 
are fixed and $\hbar\in \Bbb R$ is small, then the monodromy representation 
of $PB_r$ on ${\rm Hom}_\g(V_0,V_1\otimes...\otimes V_r)$ is unitary. The same argument can be used for the multiparameter KZ connection 
considered in \ref{kz3}. 

Similarly, as noted above, the monodromy of the Casimir connection discussed in \ref{casconn} is given by the quantum Weyl group operators of $U_q(\g)$ on the $q$-deformation $V_q$ of a finite-dimensional $\g$-module $V$. This module carries a  contravariant Hermitian form which is preserved by this action and is positive definite if $|q|=1$ and $q$ is close to $1$. Thus in this case the monodromy representation is also unitary. 

Finally, the monodromy of the rational Dunkl connection discussed in \ref{ratdun} is given by representations of the Hecke algebra $H_q(W)$, where $q=e^{2\pi ic}$, which are unitary for small real $c$. 

Using this approach, one can establish generic semisimplicity for many other examples considered above. We omit the details.

\end{document}